\documentclass{amsart}

\newcommand{\apref}[3]{\hyperref[#2]{#1\ref*{#2}#3}}

\usepackage{enumerate}
\usepackage[latin1]{inputenc}
\usepackage{dsfont}
\usepackage{amssymb,amsthm,amsmath}
\usepackage{mathtools}

\input{xy}
\xyoption{all}

\usepackage{mathrsfs}

\theoremstyle{plain}
\newtheorem{mainthm}{Theorem}

\newtheorem{mainprop}[mainthm]{Proposition}
\newtheorem{property}{Property}
\newtheorem{prop}{Proposition}[section]
\newtheorem{lemma}[prop]{Lemma}

\newtheorem{thm}[prop]{Theorem}

\newtheorem{cor}[prop]{Corollary}

\theoremstyle{definition}

\newtheorem{example}[prop]{Example}

\theoremstyle{remark}
\newtheorem{remark}[prop]{Remark}

\newcommand{\structure}{\mc S}
\DeclareMathOperator{\arcosh}{arcosh}
\newcommand{\Leb}{\textnormal{Leb}}
\DeclareMathOperator{\Unit}{U}

\DeclareMathOperator{\jc}{jc}

\newcommand{\hyp}{{\rm h}}
\newcommand{\prim}{{\rm p}}

\newcommand{\TO}{\mc L}

\newcommand{\Dfunc}{\alpha}

\DeclareMathOperator{\End}{End}

\DeclareMathOperator{\Op}{Op}

\DeclareMathOperator{\GL}{GL}

\DeclareMathOperator{\SL}{SL}
\DeclareMathOperator{\PSL}{PSL}
\DeclareMathOperator{\PGL}{PGL}

\DeclareMathOperator{\SO}{SO}
\DeclareMathOperator{\PSO}{PSO}

\DeclareMathOperator{\diag}{diag}

\DeclareMathOperator{\Tr}{Tr}
\DeclareMathOperator{\tr}{tr}

\DeclareMathOperator{\Rea}{Re}

\DeclareMathOperator{\Ind}{Ind}

\DeclareMathOperator{\bd}{bd}

\DeclareMathOperator{\Per}{Per}

\newcommand{\st}{\text{st}}

\newcommand\N{\mathbb{N}}
\newcommand\Q{\mathbb{Q}}
\newcommand\R{\mathbb{R}}
\newcommand\Z{\mathbb{Z}}
\newcommand\C{\mathbb{C}}

\newcommand{\h}{\mathbb{H}}

\newcommand{\mc}[1]{\mathcal #1}

\newcommand{\wt}{\widetilde}
\newcommand{\wh}{\widehat}

\newcommand{\eps}{\varepsilon}

\DeclareMathOperator{\id}{id}

\DeclareMathOperator{\Fct}{Fct}

\newcommand{\mat}[4]{\begin{pmatrix} #1&#2\\#3&#4\end{pmatrix}}
\newcommand{\bmat}[4]{\begin{bmatrix} #1&#2\\#3&#4\end{bmatrix}}
\newcommand{\textmat}[4]{\left(\begin{smallmatrix} #1&#2 \\ #3&#4
\end{smallmatrix}\right)}
\newcommand{\textbmat}[4]{\left[\begin{smallmatrix} #1&#2 \\ #3&#4
\end{smallmatrix}\right]}

\usepackage[colorlinks,breaklinks]{hyperref}

\begin{document}

\title[Selberg zeta functions with non-unitary twists]{Meromorphic continuation 
of Selberg zeta functions with twists having non-expanding cusp monodromy}
\author[K.\@ Fedosova]{Ksenia Fedosova}
\address{KF: Albert-Ludwigs-Universit\"at Freiburg, Mathematisches Institut, Eckerstr. 1, 79104 Freiburg 
im Breisgau, Germany}
\email{ksenia.fedosova@math.uni-freiburg.de}
\author[A.\@ Pohl]{Anke Pohl}
\address{AP: University of Bremen, Department 3 -- Mathematics, Bibliothekstr.\@ 
5,  28359 Bremen, Germany}
\email{apohl@uni-bremen.de}
\begin{abstract} 
We initiate the study of Selberg zeta functions $Z_{\Gamma,\chi}$ for 
geometrically finite Fuchsian groups $\Gamma$ and finite-dimensional 
representations $\chi$ with non-expanding cusp monodromy. 
We show that for all choices of $(\Gamma,\chi)$, the Selberg zeta function 
$Z_{\Gamma,\chi}$ converges on some half-plane in~$\C$. In addition, under the 
assumption that $\Gamma$ admits a strict transfer operator approach, we show 
that $Z_{\Gamma,\chi}$ extends meromorphically to all of $\C$.
\end{abstract}
\maketitle

\section{Introduction}

Let $\Gamma$ be a geometrically finite Fuchsian group, let 
$\chi\colon\Gamma\to\Unit(V)$ be a unitary representation of $\Gamma$ on a 
finite-dimensional unitary space $V$, and let $\h$ denote the hyperbolic plane. 
For $s\in\C$, $\Rea s \gg 1$, the (twisted) Selberg zeta function 
$Z_{\Gamma,\chi}$ for $(\Gamma,\chi)$ takes the form
\begin{equation}\label{SZF_geom}
 Z_{\Gamma,\chi}(s) = \prod_{\wh\gamma} \prod_{k=0}^\infty \det\left( 1 - 
\chi(g_{\wh\gamma}) e^{-(s+k)\ell(\wh\gamma)} \right),
\end{equation}
where $1$ denotes the identity map on $V$, $\wh\gamma$ in the first product 
ranges over all primitive periodic geodesics of $\Gamma\backslash\h$, the length 
of $\wh\gamma$ is denoted by $\ell(\wh\gamma)$, and $g_{\wh\gamma}$ is a 
primitive hyperbolic element in $\Gamma$ that is associated to $\wh\gamma$. We 
refer to Section~\ref{preliminaries} below for more details and an alternative, 
more algebraic version of \eqref{SZF_geom}.

For various combinations\footnote{To our knowledge, the cases treated in the 
literature are 
\begin{itemize}
\item $\chi$ is the trivial one-dimensional representation and $\Gamma$ is any 
geometrically finite Fuchsian group, 
\item $\Gamma$ is cofinite and $\chi$ is any unitary representation,
\item some combinations of non-cofinite $\Gamma$ (e.\,g.\@ Fuchsian Schottky 
groups) and non-trivial unitary representations $\chi$.
\end{itemize}
Even though some combinations $(\Gamma,\chi)$ with $\chi$ being unitary seem to 
be omitted in the literature, we suppose that these omissions can easily be 
handled with the existing methods.} $(\Gamma,\chi)$ the infinite product 
\eqref{SZF_geom} is known to admit a meromorphic continuation to all of $\C$, see \cite{Selberg, 
Venkov_book, Guillope} and the more extensive references further below.

Selberg zeta functions (i.\,e., the infinite products \eqref{SZF_geom} and their 
meromorphic continuations) have many applications in various areas of 
mathematics, in particular in spectral theory, harmonic analysis, number theory, 
and in mathematical physics. The arguably most important property is that the 
zeros of the Selberg zeta function $Z_{\Gamma,\chi}$ are closely related to the 
eigenvalues and resonances of the Laplace--Beltrami operator $\Delta$ as acting 
on spaces of $(\Gamma,\chi)$-automorphic functions on $\h$ of certain 
regularity. This relation allows for rich results on the (non-)existence and 
distribution of the spectral parameters of $\Delta$, and it establishes a link 
between certain geometric properties (e.\,g., geodesic length spectrum, volume, 
number of cusps; `classical mechanical' aspects) and some spectral properties 
(e.\,g., eigenvalues, scattering resonances; `quantum' aspects) of $\Gamma\backslash\h$.

It is natural to ask to which extent these results can be generalized if in 
\eqref{SZF_geom} the representation $\chi$ is allowed to be non-unitary. First 
investigations in this direction have been conducted by M\"uller \cite{Mueller}, 
Monheim \cite{Monheim}, Deitmar--Monheim \cite{Deitmar_Monheim_traceformula, 
Deitmar_Monheim_eisenstein}, Deitmar \cite{Deitmar_locally_compact}, Spilioti 
\cite{Spilioti2015} and Fedosova \cite{Fedosova_nonunitary}. 

M\"uller \cite{Mueller} considers \textit{compact} Riemannian 
locally symmetric spaces $\Gamma\backslash G/K$ (here, $G$ is a connected real 
semisimple Lie group of non-compact type with finite center, $K$ is a maximal 
compact subgroup of $G$, and $\Gamma$ is a \textit{torsion-free} 
\textit{cocompact} lattice in $G$) and arbitrary finite-dimensional 
representations $\chi\colon \Gamma\to \GL(V)$. In addition, he allows unitary 
twists of $K$, which we do not discuss here. He establishes an analogue of the 
Selberg \textit{trace} formula (which, for unitary representations, is known to 
be closely related to the Selberg zeta function) where the flat twisted 
connection Laplace operator $\Delta_\chi^\#$ takes the role of the Laplacian 
$\Delta$. Under the same conditions on $G$, $\Gamma$ and $\chi$, Deitmar and 
Monheim \cite{Deitmar_Monheim_traceformula} (see also Monheim's dissertation 
thesis \cite{Monheim}) provide a non-unitary Selberg trace formula for 
$\Gamma\backslash G$. In \cite{Deitmar_locally_compact} Deitmar generalizes 
these results to arbitrary locally compact groups~$G$. 

In \cite{Deitmar_Monheim_eisenstein, Monheim} Deitmar and Monheim consider 
hyperbolic surfaces $\Gamma\backslash\h$ of \textit{finite area} (thus, 
$G=\PSL_2(\R)$, $K=\PSO(2)$ and $\Gamma$ is a \textit{cofinite}, not necessarily 
cocompact Fuchsian group) and finite-dimensional representations 
$\chi\colon\Gamma\to\GL(V)$ that are \textit{unitary in the cusps}, i.\,e., for 
each parabolic element $p\in\Gamma$, the endomorphism $\chi(p)$ is unitary. They 
study convergence and meromorphic continuability for $\chi$-twisted Eisenstein series.

Spilioti \cite{Spilioti2015} considers \textit{compact} hyperbolic manifolds 
$\Gamma\backslash\h^n$ of \textit{odd dimension} $n$ (thus, $G=\SO(n,1)$, 
$K=\SO(n)$, $n$ odd, and $\Gamma$ is a \textit{torsion-free} \textit{cocompact} 
lattice in $G$) and arbitrary finite-dimensional representations $\chi\colon 
\Gamma\to \GL(V)$. Similar to M\"uller, she allows an additional unitary twist 
of the central elements in $K$. She shows convergence of the twisted (Ruelle and) 
Selberg zeta functions on certain half-planes in $\C$. Taking advantage of 
M\"uller's twisted Selberg trace formula she then proves meromorphic 
continuability of the twisted Selberg zeta functions to all of $\C$ and provides 
a spectral interpretation of their singularities. Fedosova 
\cite{Fedosova_nonunitary} generalizes these results to arbitrary (i.\,e., not necessarily torsion-free) cocompact 
lattices $\Gamma$ in~$G$.

In this paper we set out to study Selberg zeta functions $Z_{\Gamma,\chi}$ for 
arbitrary geometrically finite Fuchsian groups $\Gamma$ (including the 
non-cofinite ones) and finite-dimensional  representations $\chi\colon 
\Gamma\to\GL(V)$ which satisfy that for each parabolic element $p\in\Gamma$ each 
eigenvalue of the endomorphism $\chi(p)$ has absolute value $1$.

This property of $\chi$ was coined `non-expanding cusp monodromy' by Eskin, 
Kontsevich, M\"oller, and Zorich \cite{EKMZ} who prove lower bounds for Lyapunov 
exponents of flat bundles on curves, where the flat bundles are determined by 
such representations. Representations with this property also appeared in investigations of vector-valued and generalized modular forms \cite{Knopp_Mason_definition, Knopp20032, Knopp_Mason_illinois, Knopp2011, Knopp_Mason2012, Kohnen_Mason, Kohnen_Martin, Kohnen, Kohnen_Mason2}.

For cocompact lattices $\Gamma$ the class of representations with non-expanding 
cusp monodromy 
is identical to that of \textit{all} finite-dimensional representations of $\Gamma$. For 
non-cocompact Fuchsian groups this class subsumes all unitary representations, all representations that are 
unitary at cusps, all representations that are unipotent on all parabolic elements of $\Gamma$, as well as the restrictions to $\Gamma$ of the 
finite-dimensional representations of $\SL_2(\R)$, and it contains additional 
representations. We refer to Section~\ref{sec:repr} below for a more detailed 
discussion. 

In this generality (in particular with regard to the class of representations), 
the convergence of $Z_{\Gamma,\chi}$ has not been studied before, and the 
classical proof methods in their current state of art do not apply. Therefore, 
as our first main result we will establish its convergence on some right half-plane 
of $\C$ (see Theorem~\ref{thm:selberg_conv} below).

\begin{mainthm}\label{thmAnn}
Let $\Gamma$ be a geometrically finite Fuchsian group and 
$\chi\colon\Gamma\to\GL(V)$ a finite-dimensional representation with 
non-expanding cusp monodromy. Then the infinite product~\eqref{SZF_geom} 
converges for $\Rea s \gg 1$.
\end{mainthm}

In addition we will prove that Theorem~\ref{thmAnn} is sharp as stated, it cannot be generalized to a larger class of finite-dimensional representations.

\begin{mainprop}[Proposition~\ref{prop:thm_sharp} below]\label{propBnn}
Let $\Gamma$ be a geometrically finite Fuchsian group, and let $\chi\colon\Gamma\to\GL(V)$ be a finite-dimensional representation of $\Gamma$ that does not have non-expanding cusp monodromy. Then the infinite product~\eqref{SZF_geom} does not converge absolutely for any $s\in\C$.
\end{mainprop}

The combination of Theorem~\ref{thmAnn} and Proposition~\ref{propBnn} shows that the ultimate realm of finite-dimensional twists for Selberg zeta functions for Fuchsian groups (without using any additional type of regularization) is that of finite-dimensional representations with non-expanding cusp monodromy. Beyond this class, not even the first step in the definition of a twisted Selberg zeta function is tenable. 

Other interesting directions of research are those of Selberg zeta functions with twists by \emph{infinite-dimensional} representations, or for non-Fuchsian groups, or the extendability of the realm of twists by additional regularizations. These, however, will not be investigated here.

A crucial part---of independent interest---of the proof of Theorem~\ref{thmAnn} is to show 
that for every hyperbolic element $h\in\Gamma$, the trace of $\chi(h)$ grows at 
most exponentially with the length $\ell(h)$ of the periodic geodesic on 
$\Gamma\backslash\h$ associated to $h$ (or, equivalently, at most polynomially 
with the norm $N(h)$ of $h$).

\begin{mainprop}[Proposition~\ref{est_trace} below]\label{propCnn}
There exists $c\geq0$ such that uniformly for all hyperbolic elements 
$h\in\Gamma$, 
\[
 |\tr\chi(h)| \ll e^{c\ell(h)},
\]
and, even stronger,
\[
 \inf_{g\in\Gamma} \| \chi(ghg^{-1})\| \ll e^{c\ell(h)}.
\]
\end{mainprop}

For cocompact $\Gamma$, Proposition~\ref{propCnn} follows from the Lipschitz-equivalence of 
the metric on $\Gamma$ induced by the Riemannian metric on $\PSL_2(\R)$ and 
the word metric of $\Gamma$ (induced by any finite set of generators of 
$\Gamma$), see \cite{LMR, Spilioti2015}. If $\Gamma$ is cofinite and $\chi$ is 
uniform at the cusps, then Proposition~\ref{propCnn} has essentially already been proven in 
\cite{Deitmar_Monheim_eisenstein} (using \cite[Propositions~2.7 and 
2.9]{Deitmar_Monheim_eisenstein} in place of \eqref{est_nonparab} below). For 
the general case of arbitrary geometrically finite Fuchsian groups $\Gamma$ and 
finite-dimensional representations $\chi$ with non-expanding cusp monodromy, the 
proofs in \cite{LMR, Spilioti2015, Deitmar_Monheim_eisenstein} do not apply 
anymore. For this generality, it is essential to understand which weight $\chi$ 
attributes to  periodic geodesics which travel `high into the cusps' (long cusp 
excursions) or, equivalently, to understand the contribution of parabolic 
elements to the growth of $\chi$. In Section~\ref{sec:conv} below we will provide a 
detailed study, including a kind of logarithm law (see Lemma~\ref{jordanblock} 
below).

Proposition~\ref{propCnn} allows us to bound the \textit{twisted} Selberg zeta function 
$Z_{\Gamma,\chi}$  by a shift of the \textit{non-twisted} Selberg zeta function 
$Z_{\Gamma,\bf 1}$ for the trivial one-dimensional representation~$\bf 1$. The 
well-known convergence properties of $Z_{\Gamma,\bf 1}$ then yield Theorem~\ref{thmAnn}. 

The proof of Proposition~\ref{propCnn}, and hence of Theorem~\ref{thmAnn}, is uniform for all pairs 
$(\Gamma,\chi)$ of geometrically finite Fuchsian groups $\Gamma$ and 
finite-dimensional representations $\chi$ with non-expanding cusp monodromy.

Our second main result concerns the meromorphic continuability of the infinite 
product~\eqref{SZF_geom}. As already mentioned above, prior to this paper, 
meromorphic continuability of $Z_{\Gamma,\chi}$ has been known for 
\textit{unitary} representations only (note that \cite{Spilioti2015, 
Fedosova_nonunitary} consider odd-dimensional spaces only, thus they do not 
treat Fuchsian groups). For these investigations several methods (e.\,g.\@ trace 
formulas, microlocal analysis, transfer operator techniques) have been employed 
and applied to various classes of $(\Gamma,\chi)$, resulting in alternative or 
complementary proofs of meromorphic extendability in different generality. In 
particular the usage of trace formulas and microlocal analysis requires to 
restrict the uniformity of the considerations to certain classes of Fuchsian 
groups, e.\,g., cocompact or cofinite Fuchsian groups or Fuchsian Schottky 
groups. We refer to \cite{Selberg, Venkov_book, Patterson, Sullivan, Guillope, 
juhl2012cohomological, DyatlovGuillarmou, Ruelle_zeta, Fried_zetafunctionsI, 
Mayer_thermo, Mayer_thermoPSL, Bunke_Olbrich, gon2010zeta, DZ} for examples and 
more details.

The proof of meromorphic continuability for Selberg zeta functions with 
non-unitary twists in  \cite{Spilioti2015, Fedosova_nonunitary} uses harmonic 
analysis on symmetric spaces, in particular trace formulas and orbital 
integrals. Due to this approach their results are currently restricted to 
(compact) hyperbolic spaces/orbifolds of \textit{odd} dimension, a major problem 
being to guarantee that residues are integral. It would certainly be interesting 
to see if these methods can be adapted to even-dimensional spaces.

Also the results in \cite{Mueller, Deitmar_Monheim_traceformula, Monheim} on the 
existence of Selberg trace formulas twisted by non-unitary representations (for 
compact spaces) or by representations that are unitary at the cusps (for 
non-compact hyperbolic surfaces) rely on harmonic analysis of symmetric spaces. 
It would certainly be interesting to understand if these trace formulas integrate to Selberg zeta 
functions and to which extent they generalize to non-compact spaces, spaces of 
non-finite volume as well as beyond representations that are unitary at the 
cusps.

Nevertheless, in order to show the meromorphic continuability of 
\eqref{SZF_geom}, we here will develop transfer operator techniques which permit a 
uniform approach for all admissible combinations $(\Gamma,\chi)$. In particular, 
in contrast to some other methods, we will not need to distinguish between 
Fuchsian groups with and without cusps, or cocompact, cofinite and non-cofinite 
ones. 

However, we will suppose that $\Gamma$ admits a \textit{strict transfer operator 
approach}. This means, roughly, that there exists a transfer operator family 
$\TO_{\Gamma,s}$ such that 
\begin{equation}\label{fredholm_simple}
 Z_{\Gamma,\bf 1}(s) = \det\left(1 - \TO_{\Gamma,s}\right),\qquad \Rea s \gg 1.
\end{equation}
In other words, the Selberg zeta function $Z_{\Gamma,\bf 1}$ for the \textit{trivial 
one-dimensional} representation~$\bf 1$ is represented by the Fredholm 
determinant of a transfer operator family on some right half-plane in $\C$ where 
$Z_{\Gamma,\bf 1}$ is given by the infinite product~\eqref{SZF_geom}. We refer 
to Section~\ref{def_stoa} below for a precise definition.

In view of the already existing strict transfer operator approaches and the 
different methods for their construction \cite{Ruelle_zeta, 
Fried_zetafunctionsI, Artin, Series, Mayer_thermo, Mayer_thermoPSL, Bowen, 
Pollicott, Bowen_Series, Morita_transfer, Fried_triangle,  Patterson_Perry, 
Guillope_Lin_Zworski, Chang_Mayer_extension, Manin_Marcolli_transfer, 
Pohl_Symdyn2d, Moeller_Pohl, Pohl_representation}, it might well be that this 
requirement is not a severe restriction on $\Gamma$ at all. Moreover, it is most 
likely that with the methods we propose in this article the meromorphic 
continuability of \eqref{SZF_geom} can also be shown starting with \textit{non-strict} 
transfer operator approaches as provided in \cite{Pollicott, Morita_transfer, 
Mayer_Muehlenbruch_Stroemberg} (for certain classes of cofinite Fuchsian 
groups). In the latter case, the representation of $Z_{\Gamma,\bf 1}$ is of the 
form 
\[
 Z_{\Gamma,\bf 1}(s) = h(s) \det\left( 1 - \TO_{\Gamma,s}\right),\qquad \Rea s 
\gg 1
\]
where $h$ is a meromorphic function accounting for certain non-exact codings. A 
\textit{weak version} of our second main result reads as follows.

\begin{mainthm}\label{thmDnn}
Let $\Gamma$ be a geometrically finite Fuchsian group which admits a strict 
transfer operator approach, and let $\chi$ be a finite-dimensional 
representation of $\Gamma$ with non-expanding cusp monodromy. Then the infinite 
product~\eqref{SZF_geom} for $Z_{\Gamma,\chi}$ extends meromorphically to all of 
$\C$. The poles of $Z_{\Gamma,\chi}$ are contained in $\tfrac12(d_1-\N_0)$, 
where $d_1$ denotes the maximal length of a Jordan chain of $\chi(p)$ with eigenvalue $1$, 
$p\in\Gamma$ parabolic.
\end{mainthm}

In fact, we will prove a stronger version of Theorem~\ref{thmDnn} (see 
Theorem~\ref{mainthm:strong} below) for which we first show that the infinite 
product \eqref{SZF_geom} is represented by a twisted version of the transfer 
operator family:
\begin{equation}\label{fredholm_twisted}
 Z_{\Gamma,\chi}(s) = \det\left(1-\TO_{\Gamma,\chi,s}\right),\qquad \Rea s\gg 1.
\end{equation}
Then we will prove that not only the Selberg zeta function $Z_{\Gamma,\chi}$ and the Fredholm determinant $s\mapsto 
\det\left(1-\TO_{\Gamma,\chi,s}\right)$ admit a meromorphic continuation but 
also the map
\begin{equation}\label{TO_meromcont}
 s\mapsto \TO_{\Gamma,\chi,s}
\end{equation}
itself. In addition, we will provide upper bounds on the order of the poles.

The requirement that $\Gamma$ admit a strict transfer operator approach does not 
directly ask for the equality~\eqref{fredholm_simple} but for the existence of a 
suitable (uniformly expanding) iterated function system (IFS). If the representation $\chi$ is \textit{unitary}, then establishing 
\eqref{fredholm_simple}, the meromorphic continuation of \eqref{TO_meromcont} 
and (the strong version of) Theorem~\ref{thmDnn} starting from such an IFS is by now 
standard, see, e.\,g., \cite{Ruelle_zeta, Fried_zetafunctionsI, Mayer_thermo, 
Pohl_representation}. 

Hence, the focus in the proof of Theorem~\ref{thmDnn} is to accommodate 
\textit{non-unitary} representations. For \eqref{fredholm_twisted}, nuclearity 
of the twisted transfer operators is indispensable. To show this  property we will
rely on the logarithm law for cusp excursions as well as on the finiteness of 
the number of cusps of $\Gamma$. 

If $\chi$ is the trivial one-dimensional representation, then the standard proof of 
the meromorphic continuation of \eqref{TO_meromcont} takes advantage of the 
Hurwitz zeta function. For any unitary representations $\chi$, Pohl 
\cite{Pohl_representation} recently showed that this proof can be adapted by 
combining a diagonalization of $\chi(p)$, $p\in\Gamma$ parabolic, with the 
Lerch zeta function. For generic representations $\chi$ with non-expanding cusp 
monodromy, the endomorphisms $\chi(p)$, $p\in\Gamma$ parabolic, are not 
necessarily diagonalizable anymore. However, as we will show in 
Section~\ref{sec:meromorphic} below, a careful use of the Lerch transcendent 
and its derivatives in combination with Jordan normal forms of $\chi(p)$, 
$p\in\Gamma$ parabolic, allows us to establish meromorphic continuability.

As an additional result to those already mentioned, we will briefly discuss, in Section~\ref{sec:factor} below, that the 
Venkov--Zograf factorization formulas also hold for Selberg zeta functions with 
twists by representations with non-expanding cusp monodromy.

The next natural goal in this line of research is to seek for a spectral and 
topological characterization of the zeros of the Selberg zeta function 
$Z_{\Gamma,\chi}$. Answering this question we leave for future research. The 
first step in this direction is already done in \cite{FP_Eisenstein}, where the 
convergence of $\chi$-twisted Eisenstein series is addressed.

\subsubsection*{Acknowledgement} The authors wish to thank the Max Planck 
Institute for Mathematics in Bonn for financial support and excellent working 
conditions during the preparation of this manuscript. Further, AP acknowledges 
support by the DFG grant PO 1483/2-1. Both authors thank Roelof Bruggeman for his interest and his helpful remarks on this manuscript.

\section{Preliminaries}
\label{preliminaries}

\subsection{Hyperbolic geometry}
We use the upper half-plane model of the hyperbolic plane
\[
\h \coloneqq \{ x+iy \mid y > 0\}
\]
endowed with the hyperbolic metric
\begin{equation}\label{Riemannmetric}
ds^2 = \frac{dx^2 + dy^2}{y^2}.
\end{equation}
The group of 
orientation-preserving Riemannian isometries of $\h$ is isomorphic to
the group $\PSL_2(\R) = \SL_2(\R)/\{\pm \id\}$, considered as acting by 
fractional linear transformations on $\h$. 

By continuity, this action extends to the geodesic boundary $\partial\h$ of 
$\h$. If we identify $\partial\h$ with $\wh\R\coloneqq \R\cup\{\infty\}$ (or 
$P^1(\R)$) in the standard way, then the action on the geodesic closure
\[
 \overline{\h} = \h \cup \partial\h
\]
of $\h$ reads as follows: For $g=\textbmat{a}{b}{c}{d}\in \PSL_2(\R)$ and $z\in 
\overline{\h}$ we have
\[
g.z \coloneqq 
\begin{cases} 
\infty & \text{if $z=\infty$, $c=0$, or $z\not=\infty$, $cz+d=0$}
\\
\frac{a}{c} & \text{if $z=\infty$, $c\not=0$}
\\
\frac{az+b}{cz+d} & \text{otherwise.}
\end{cases} 
\]

For any hyperbolic element $h\in\PSL_2(\R)$ (i.\,e., $|\tr(h)|>2$) we let $N(h)$ 
denote the norm of $h$, thus
\begin{equation}\label{norm_h}
 N(h) \coloneqq \max\left\{ \lambda^{2} \ \left\vert\ \text{$\lambda$ 
eigenvalue of $h$} \vphantom{|\lambda|^{\frac12}}\right.\right\}.
\end{equation}

\subsection{Hyperbolic surfaces/orbifolds and periodic geodesics}
Let $\Gamma$ be a Fuchsian group. We denote by $[\Gamma]_\hyp$ the set of 
$\Gamma$-conjugacy classes of hyperbolic elements in $\Gamma$, and by 
$[\Gamma]_\prim$ the subset of $\Gamma$-conjugacy classes of primitive 
hyperbolic elements. For $h\in\Gamma$ hyperbolic, $[h]$ denotes the 
corresponding element in $[\Gamma]_\hyp$. 

Let  
\[
X \coloneqq \Gamma \backslash\h
\]
denote the (two-dimensional, connected, real hyperbolic, good) orbifold with fundamental 
group~$\Gamma$. As is well-known, the periodic geodesics on~$X$ are in bijection with~$[\Gamma]_\prim$. A natural isomorphism between periodic geodesics and~$[\Gamma]_\prim$, on which we will rely in 
Section~\ref{sec:conv} below, is given as follows: To a hyperbolic element 
$h\in\Gamma$ we assign the geodesic $\gamma_h$ on $\h$ for which 
$\gamma_h(+\infty)$ is the attracting fixed point of $h$, and $\gamma_h(-\infty)$ is the repelling fixed point. Note that $\gamma_{h^n} = \gamma_h$ 
for any $n\in\N$. Then we identify an equivalence class $[g]\in [\Gamma]_\prim$ of a \textit{primitive} hyperbolic element $g\in\Gamma$
with the periodic geodesic $\wh\gamma_g$ on $X$ for which $\gamma_g$ is a 
representing geodesic on $\h$. In this case, the (primitive) length 
$\ell(\wh\gamma_g)$ of $\wh\gamma_g$ is given by
\begin{equation}\label{lengthpergeod}
 \ell(\wh\gamma_g) = \log N(g).
\end{equation}
We remark that $\wh\gamma_g$ and $N(g)$ do not depend on the chosen representative $g\in\Gamma$ of $[g]\nobreak\in\nobreak[\Gamma]_\prim$. If we consider (as we shall) periodic geodesics on $X$ with length 
multiplicities, then this identication extends to all of $[\Gamma]_\hyp$. 

For any hyperbolic $h\in\Gamma$ we set
\[
 \ell(h) = \log N(h).
\]

\subsection{Selberg zeta function}\label{sec:szf}
Let $\Gamma$ be a geometrically finite Fuchsian group, and let $\chi\colon 
\Gamma\to \GL(V)$ be a finite-dimensional representation of $\Gamma$. The 
Selberg zeta function $Z_{\Gamma,\chi}=Z(\cdot, \Gamma,\chi)$ for 
$(\Gamma,\chi)$ is formally defined by 
\begin{equation}\label{pre_def}
 Z_{\Gamma,\chi}(s) \coloneqq Z(s,\Gamma,\chi) \coloneqq \prod_{[g]\in [\Gamma]_\prim} 
\prod_{k=0}^\infty \det\left( 1- \chi(g) N(g)^{-(s+k)}\right),
\end{equation}
where $1$ denotes the identity map on $V$. The relation to the 
definition in~\eqref{SZF_geom} in the Introduction is explained by 
\eqref{lengthpergeod}.

If $\chi$ is the trivial one-dimensional representation ${\bf 1}$, then the 
infinite product~\eqref{pre_def} is known to converge for $\Rea s \gg 1$ (more 
precisely, for $\Rea s> \delta$, where $\delta$ denotes the Hausdorff dimension 
of the limit set of $\Gamma$) \cite{Selberg, Venkov_book, Guillope, Patterson, 
Sullivan}. As recalled in the Introduction, convergence is more generally known 
if $\chi$ is unitary---a result we do not take advantage of in this article. 

For the elementary Fuchsian groups, the Selberg zeta function is most 
interesting for the hyperbolic funnel groups, that is, those Fuchsian groups 
that are generated by a single hyperbolic element. For the other elementary 
Fuchsian groups, the first product in \eqref{pre_def} is void, thus the Selberg 
zeta function is constant $1$ by convention.

\subsection{Representations}
Let $\Gamma$ be a Fuchsian group. A finite-dimensional (real) representation 
$\chi\colon\Gamma \to \GL(V)$
is said to have \textit{non-expanding cusp monodromy} if for each parabolic 
element $p\in\Gamma$, each eigenvalue of the endomorphism $\chi(p)$ has absolute 
value $1$.

\subsection{Further notations}

We use $\N = \{1,2,3,\ldots\}$ to denote the set of natural numbers (without 
zero), and let $\N_0 \coloneqq \N\cup\{0\}$. Further, for functions $f\colon M\to \C$, 
$g\colon M\to\R$ on some set $M$ we write $|f|\ll g$ or $f\ll g$ or $|f(m)|\ll 
g(m)$ or $f(m)\ll g(m)$ for 
\[
\exists\, C>0\ \forall\, m\in M\colon |f(m)|\leq C g(m).
\]
If the constant $C$ depends on an exterior variable, say $p$, and we need to 
keep track of this dependence ,then we indicate this variable in an subscript, 
e.\,g., $\ll_p$. Analogously, we use $\gg$ and $\gg_p$.

\section{Convergence of twisted Selberg zeta functions}\label{sec:conv}

Throughout this section let $\Gamma$ be a geometrically finite Fuchsian group 
and 
\[
\chi\colon \Gamma\to\GL(V)
\]
a finite-dimensional representation with non-expanding cusp monodromy. We will show 
that the infinite product in \eqref{pre_def} converges if $\Rea s$ is 
sufficiently large (depending on $\Gamma$ and $\chi$).

\begin{thm}\label{thm:selberg_conv}
There exists $C=C(\Gamma,\chi)\in\R$ such that the infinite product in 
\eqref{pre_def} converges compactly on the right half-plane $R\coloneqq \{ s\in\C 
\mid \Rea s > C\}$. Thus, on $R$, it defines a holomorphic function without zeros.
\end{thm}

As already mentioned in the Introduction, the proof of 
Theorem~\ref{thm:selberg_conv} is based on bounding the Selberg zeta function 
$Z_{\Gamma,\chi}$ for the representation $\chi$ by the Selberg zeta function 
$Z_{\Gamma,\bf 1}$ for the trivial one-dimensional representation $\bf 1$, and 
then taking advantage of the fact that the convergence of $Z_{\Gamma,\bf 1}$ is 
well-known. More precisely, we find $c\in\R$ such that for all $s\in\C$ with 
$\Rea s$ sufficiently large we have
\begin{align*}
 \log Z(s,\Gamma,\chi) \ll |\log Z(\Rea s - c, \Gamma, {\bf 1})|.
\end{align*}
Proposition~\ref{est_trace} below is the key step for the proof that the shift 
in the argument of $Z_{\Gamma,{\bf 1}}$ is indeed uniform. This proposition 
makes crucial use of $\chi$ having non-expanding cusp monodromy.

Throughout we fix a norm $\|\cdot\|$ on $\GL(V)$. Since the trace norm is 
sub-multiplicative, and all norms on $\GL(V)$ are equivalent due to 
finite-dimensionality, $\|\cdot\|$ is essentially sub-multiplicative. Thus, 
there exists $C>0$ such that for all $g,h\in\Gamma$ we have the bound
\begin{equation}\label{esssubmult}
 \|\chi(gh)\| \leq C \|\chi(g)\| \cdot \|\chi(h)\|.
\end{equation}

\begin{prop}\label{est_trace}
There exists $c\geq0$ such that uniformly for all hyperbolic elements $h\in\Gamma$, 
we have
\begin{equation}\label{eq:est_trace}
 |\tr \chi(h) | \ll e^{c \ell(h)},
\end{equation}
or, equivalently,
\[
 |\tr \chi(h)| \ll N(h)^{c}.
\]
Moreover, 
\[
 \inf_{g\in\Gamma} \| \chi(ghg^{-1})\| \ll N(h)^c.
\]
In particular, all implied constants are independent of $h$.
\end{prop}

For elementary Fuchsian groups, Proposition~\ref{est_trace} is obvious. For the 
proof of Proposition~\ref{est_trace} for non-elementary Fuchsian groups we need 
some preparations. In Section~\ref{block_bound} below we will recall an estimate of 
word lengths (or rather `block' lengths) of hyperbolic elements in terms of the 
displacement of the point $i$ of the hyperbolic plane $\h$ as provided by 
Eichler \cite{Eichler} and by Knopp and Sheingorn \cite{Knopp_Sheingorn}. This 
result allows us to bound the number of cusp excursions of a periodic geodesic. 

In Section~\ref{sec:growth} below we will provide, for every parabolic element 
$p\in\Gamma$, a polynomial bound on $\chi(p^m)$ as $m\to\infty$ (the 
logarithm-type law mentioned in the Introduction). In Section~\ref{def_r0} below 
we will recall the notion of horoballs and their relation to Dirichlet fundamental 
domains, which obviously is closely related to Siegel sets. For every periodic 
geodesic $\wh\gamma$ on $\Gamma\backslash\h$, the latter relation allows us to 
estimate how deep $\wh\gamma$ travels into the cusp represented by the parabolic 
element~$p$. This permits us to deduce bounds on the power $m$ with which $p^m$ 
appears in the word/block representation of a hyperbolic element $h$ associated 
to $\wh\gamma$. These bounds grow at most polynomial (with uniform degree) in the 
norm $N(h)$ of $h$ or, equivalently, at most exponential in the length 
$\ell(\wh\gamma)$ of $\wh\gamma$.

\subsection{Bounds on block lengths}\label{block_bound}
We recall an upper estimate of `block' lengths of hyperbolic elements $h$ in 
terms of the displacement of the point $i$ by~$h$. This estimate was proved by 
Eichler \cite{Eichler} for Fuchsian lattices and was then generalized by Knopp 
and Sheingorn \cite{Knopp_Sheingorn} to all geometrically finite Fuchsian 
groups. 

Let $D$ be a Dirichlet fundamental domain for $\Gamma$, and let $\mc 
A\subseteq\Gamma$ be the set of side-pairing elements of $D$. We present each 
element $g\in\Gamma$ as a word over $\mc A$ via the standard Morse coding 
algorithm induced by $D$. This means that the presentation $g=A_1A_2\cdots A_n$ 
as a word over $\mc A$ is deduced as follows: We fix a point $x$ in $D$ and a 
point $y$ in $g.D$ such that the geodesic $\gamma$ (on $\h$) connecting $x$ and 
$y$ stays away from the vertices of $k.D$ for all $k\in\Gamma$. Thus, when 
moving along the geodesic $\gamma$ from $x$ to $y$, one passes through a 
(unique) sequence of the form 
\[
D,\ A_1.D,\ A_1A_2.D,\ \ldots,\ A_1A_2\cdots A_n.D = g.D
\]
with $A_1,\ldots, A_n\in\mc A$. This sequence is indeed independent of the 
choice of $x$ and $y$ as long as the connecting geodesic $\gamma$ satisfies the 
condition mentioned. In each step, the element $A_i\in\mc A$ is determined by 
the side of $A_1\cdots A_{i-1}.D$ through which $\gamma$ passes, or 
equivalently, the side $S_i$ of $D$ through which $A_{i-1}^{-1}\cdots 
A_1^{-1}.\gamma$ passes. It is the unique element $A_i$ in 
$\Gamma\smallsetminus\{\id\}$ (or, equivalently, in $\mc A$) such that $A_i.S_i$ is 
again a side of $D$.

We arrange the presentation $g=A_1\cdots A_n$ as follows into `blocks': Each 
non-parabolic element $A\in\mc A$ forms a block on its own, and each locally 
maximal subword of $A_1\cdots A_n$ of the form $A^k=A\cdots A$ ($k$-times) with 
$A\in\mc A$ being parabolic and $k\in\N$ forms a block. Here, `locally maximal' 
means that 
\[
A_1\cdots A_n = A_1\cdots A_{i} A^k A_{i+k+1}\cdots A_n
\]
with $A_i\not=A\not=A_{i+k+1}$ (thus, $k$ is maximal). Let $b_D(g)$ denote the 
number of blocks of $g$. Note that $b_D$ may depend on the choice of $D$.

Let $d_\h$ denote the metric on $\h$ induced by the Riemannian 
metric~\eqref{Riemannmetric}. By Eichler \cite{Eichler} and Knopp and 
Sheingorn \cite{Knopp_Sheingorn} there exist constants $c_1,c_2>0$, 
possibly depending on the choice of $D$, such that for all \textit{hyperbolic} 
elements $h\in\Gamma$, we have
\begin{equation}\label{bound_block}
 b_D(h) \leq c_1 \cdot d_\h(i,h.i) + c_2.
\end{equation}

\subsection{Growth of norms}\label{sec:growth}

We show that for every parabolic element $p\in\Gamma$, the norm of $\chi(p^m)$ 
grows at most polynomially as $m\to\infty$. The degree of the polynomial is 
uniformly (over all parabolic elements) bounded by $\dim V -1$.

\begin{lemma}\label{jordanblock}
Let $p\in\Gamma$ be parabolic, $d_p\in\N$, and suppose that the lengths of the 
Jordan chains of $\chi(p)$ are bounded by $d_p$. Then
\[
 \| \chi(p^m) \| \ll_p m^{d_p-1}
\]
for all $m\in\N$.
\end{lemma}

\begin{proof}
We fix a basis of $V$ with respect to which $\chi(p)$ is represented in Jordan 
normal form. Let $r\in\N$ be the number of Jordan blocks, $k_1,\ldots, k_r\in\N$ 
the lengths of the Jordan chains, and $\lambda_1,\ldots, \lambda_r\in\C$ the 
corresponding eigenvalues of $\chi(p)$. Then 
\[
 \| \chi(p^m) \| \ll_p \max_{j=1}^r \max_{\ell=0}^{k_j-1} \left\{ 
|\lambda_j|^{m-\ell} {m\choose \ell} \right\}.
\]
By hypothesis, $k_j\leq d_p$ for $j=1,\ldots, r$. Further, since $\chi$ has 
non-expanding cusp monodromy, $|\lambda_j|=1$ for $j=1,\ldots, r$. Thus
\[
\max_{j=1}^r \max_{\ell=0}^{k_j-1} \left\{ |\lambda_j|^{m-\ell} {m\choose \ell} 
\right\} \leq \max_{l=0}^{d_p - 1} {m\choose l} = \begin{cases}
{m\choose d_p-1} \quad\text{for $d_p -1 < 
\left\lfloor\frac{m}{2}\right\rfloor$}, \\[3mm] 
{m \choose \left\lfloor\frac{m}{2}\right\rfloor} \quad \text{for $d_p -1 \geq 
\left\lfloor\frac{m}{2}\right\rfloor.$} 
\end{cases}
\]
Hence,
\[
\max_{j=1}^r \max_{\ell=0}^{k_j-1} \left\{ |\lambda_j|^{m-\ell} {m\choose \ell} 
\right\} \ll m^{d_p -1}.
\]
This completes the proof.
\end{proof}

\subsection{Return bounds}\label{def_r0}

Recall that $d_\h$ denotes the metric on $\h$ induced by the Riemannian 
metric~\eqref{Riemannmetric}, and let $\gamma$ be a geodesic on $\h$. Then the 
\textit{Buseman function} $b_\gamma\colon \h\to \R$ is defined by 
\[
 b_\gamma(z) \coloneqq \lim_{t\to\infty} \big( t - d_\h(z,\gamma(t)) \big).
\]
Let $c\in\partial \h$ and $r\in\R$. The \textit{horoball} $B(c,r)$ centered at 
$c$ with `radius' $r$ is the sublevel set
\[
 B(c,r) \coloneqq b_\eta^{-1}\big( [r,\infty) \big),
\]
where $\eta$ is any geodesic on $\h$ with $\eta(\infty) = c$. If $c \in 
\overline{\h}$ is cuspidal with respect to $\Gamma$, i.e., $c$ is fixed by a 
primitive parabolic element $p$ in $\Gamma$, then there exists a radius $r=r(c)$ 
such that whenever $z,w\in B(c,r)$ are $\Gamma$-equivalent, hence $g.z=w$ for 
some $g\in\Gamma$, then $g=p^m$ for some $m\in\Z$. Even more, if $D$ is a 
Dirichlet fundamental domain for $\Gamma$, $\mc A$ is the set of side-pairing 
elements, and  $\mc P\subseteq\mc A$ its subset of the parabolic elements among 
the side-pairing elements, then there exists $r_0\in\R$ such that for each 
$p\in\mc P$, 
\[
 B(c_p, r_0) \subseteq \bigcup_{m\in\Z} p^{m}.D,
\]
where $c_p$ denotes the fixed point of $p$ (see, e.\,g., 
\cite[Section~9.4]{Beardon}, in particular \cite[Theorem~9.4.5]{Beardon}). We call each 
such $r_0\in\R$ a \textit{return bound} for $D$.

\subsection{Proof of Proposition~\ref{est_trace}}
Let $D$ be a Dirichlet fundamental domain for $\Gamma$ which we fix once and for 
all.  In the following, all of the implied constants as well as the numerically 
unspecified constants may depend on the choice of $D$ but they are independent 
of any specially considered element in $\Gamma$. 

Let $\mc A$ denote the set of side-pairing elements of $D$, let 
\[
 K \coloneqq \max_{A\in\mc A} \|\chi(A)\|
\]
and recall the definition of block presentations and the block counting 
function 
\[
b\coloneqq b_D
\]
from Section~\ref{block_bound}. Further let $D_0$ denote the (compact) subset of 
$\overline{D}$ which represents the compact core of $\Gamma\backslash\h$ (or, 
more generally, let $D_0$ be any compact subset of $\overline{D}$ which 
represents a compact subset of $\Gamma\backslash\h$ that is intersected by all 
periodic geodesics).

Since every periodic geodesic on $\Gamma\backslash\h$ intersects the compact 
core, every $\Gamma$-conjugacy class of hyperbolic elements in $\Gamma$ contains 
at least one representative $h$ whose associated geodesic $\gamma_h$ on $\h$ 
intersects $D_0$. It suffices to prove \eqref{eq:est_trace} for one such 
representative $h$ out of each $\Gamma$-conjugacy class of hyperbolic elements 
since $\tr\circ\chi$ as well as the length function $\ell$ are constant on these 
conjugacy classes.

Let $h\in\Gamma$ be a hyperbolic element such that its associated geodesic 
$\gamma_h$ on $\h$ intersects $D_0$. Suppose that 
\[
 h=h_1\cdots h_{b(h)}
\]
is its block presentation. Let
\[
 J_p \coloneqq \big\{ j\in \{1,\ldots, b(h)\}\ \big\vert\ \text{$h_j$ parabolic} 
\big\} 
\]
be the set of indices corresponding to the parabolic blocks, and set
\[
 J_n \coloneqq \{1,\ldots, b(h)\}\smallsetminus J_p.
\]
Further set 
\[
 b_p \coloneqq \# J_p \quad\text{and}\quad  b_n \coloneqq \# J_n,
\]
so that $b_p + b_n = b(h)$. Then
\begin{align}\label{tr_est1}
 | \tr \chi(h)| \ll \| \chi(h) \|  & \leq C^{b(h)-1} \prod_{j=1}^{b(h)} \| \chi(h_j) 
\|  
\\
& = C^{b(h)-1}\left( \prod_{i\in J_n} \| \chi(h_i)\| \right)\left( \prod_{j\in J_p} \| 
\chi(h_j)\| \right) \nonumber
\end{align}
with $C$ as in \eqref{esssubmult}. For each $i\in J_n$, the block $h_i$ consists 
of a single element in $\mc A$. Hence
\begin{equation}\label{est_nonparab}
 \prod_{i\in J_n} \| \chi(h_i)\|  \leq K^{b_n}.
\end{equation}

For each $k\in J_p$, the block $h_j$ equals $p_j^{m_j}$ for some parabolic 
element $p_j\in\mc A$ and $m_j\in\N$. Let $d_0$ denote the maximal length of a 
Jordan chain of $\chi(p)$, $p\in\mc A$ parabolic. Further let $\tilde C>0$ be an 
upper bound for the implied constants in Lemma~\ref{jordanblock}, when applied 
to the finitely many parabolic elements in $\mc A$. Then Lemma~\ref{jordanblock} 
yields
\begin{equation}\label{est_parab}
 \prod_{j\in J_p} \| \chi(h_j)\| \leq \tilde C^{b_p} \prod_{j\in J_p} 
m_j^{d_0-1}.
\end{equation}

In order to bound $m_j$ for $j\in J_p$ we fix a point $z\in \gamma_h(\R)\cap 
D_0$. Without loss of generality, we may assume that $z$ serves to deduce the 
presentation of $h$ as a word over $\mc A$ as in Section~\ref{block_bound}. 
Thus, the geodesic segment $c_h$ connecting $z$ and $h.z$ is a segment of 
$\gamma_h$, and 
\begin{equation}\label{splitpoint}
 d_\h(z, h.z) = \ell(h) = \ell(c_h).
\end{equation}
Here, $\ell(c_h)$ denotes the length of $c_h$.

Let $j\in J_p$. Let $b_j$ denote the fixed point of the parabolic element $p_j$ 
and let $\gamma_j \coloneqq h_{j-1}^{-1}\cdots h_1^{-1}.\gamma_h$. Since 
$h_j=p_j^{m_j}$, the geodesic $\gamma_j$ passed through $A.D$ for some $A\in\mc 
A\smallsetminus\{p_j\}$ before entering $D$, then passes through $D$, $p_j.D$, 
$\ldots$, $p_j^{m_j}.D$ and finally enters into $p_j^{m_j}A'.D$ for some 
$A'\in\mc A\smallsetminus\{p_j\}$. Set
\[
 D_j\coloneqq \bigcup_{k=0}^{m_j} p_j^k.D
\]
and let  
\[
 t_j \coloneqq \mu_\Leb\! \left( \big\{ t \in \R \ \big\vert\  \gamma_j(t) \in D_j \big\} 
\right)
\]
be the time (or, equivalently, the length) that $\gamma_j$ spends in $D_j$. 
Here, $\mu_\Leb$ denotes the Lebesgue measure on $\R$ (normalized such that 
$\mu_\Leb([0,1])=1$). Let $S_1,S_2$ be the two sides of $D$ which meet at $b_j$, and 
suppose that $\gamma_j$ exits $D$ through $S_2$. Then, for $k=1,\ldots, m_j-1$, 
it enters $p_j^k.D$ through $p_j^k.S_1 = p_j^{k-1}.S_2$ and exits through 
$p_j^k.S_2$. Suppose that 
\[
 a_j \coloneqq \inf_{k=1}^{m_j-1} \mu_\Leb\! \left( \big\{ t\in\R \ \big\vert\ \gamma_j(t) 
\in p_j^k.D \big\}\right)
\]
denotes the minimal time that $\gamma_j$ spends in any of the $p_j^k.D$, 
$k=1,\ldots, m_j-1$. Then
\[
 m_j -1 \leq \frac{t_j}{a_j}.
\]
In order to provide a lower bound on $a_j$, let $r_0\in\R$ be a return bound for 
$D$ (cf.\@ Section~\ref{def_r0}) and recall that 
\[
 B(b_j, r_0) \subseteq \bigcup_{n\in\Z} p_j^n.D.
\]
Obviously, the time that $\gamma_j$ spends in $B(b_j,r_0)$ is bounded by $t_j$. 
Therefore, $\gamma_j$ does not enter $B(b_j, r_0+t_j)$. Let
\[
 H \coloneqq D\cap B(b_j,r_0)\smallsetminus B(b_j,r_0+t_j)
\]
and let $\mc C$ denote the set of geodesics $\gamma$ on $\h$ that enter $H$ 
through $S_1$ and exit through $S_2$. Then $a_j$ is bounded below by
\[
 a'_j \coloneqq \inf_{\gamma\in\mc C} \mu_\Leb\!\left( \{ t\in\R\mid \gamma(t) \in 
H\}\right).
\]
Let $w_j$ be the cusp width of $D$ at $b_j$. By elementary hyperbolic geometry,
\[
 a'_j  = \arcosh\left( 1 + \frac{w_j}{2e^{2r_0}} e^{-2t_j}\right) \gg e^{-3t_j}.
\]
Therefore,
\begin{align*}
 m_j \leq \frac{t_j}{a_j} + 1 \ll t_j e^{3t_j} + 1 \ll e^{4t_j}.
\end{align*}
Since
\[
 \sum_{j\in J_p} t_j \leq \ell(c_h) = \ell(h),
\]
we find
\begin{equation}\label{est_m}
 \prod_{j\in J_p} m_j^{d_0-1} \ll e^{4(d_0-1)\ell(h)}.
\end{equation}
The combination of \eqref{tr_est1}, \eqref{est_nonparab}, \eqref{est_parab} and 
\eqref{est_m} yields
\begin{equation}
 C |\tr \chi(h) | \ll C \|  \chi(h)\| \leq C_1^{b(h)}e^{4(d_0-1) \ell(h)}
\end{equation}
for some $C_1>0$. By \eqref{bound_block},
\[
 C_1^{b(h)} \ll e^{c_1\cdot d_\h(i,h.i)}
\]
for some $c_1>0$. The triangle inequality, the left-invariance of the metric 
$d_\h$, and the compactness of $D_0$ imply the existence of  $c_2\geq0$ such that
\begin{align}
 d_\h(i, h.i) &\leq  d_\h(i, z) + d_\h(z, h.z) + d_\h(h.z, h.i)\leq 2 d_\h(i, z) 
+ d_\h(z, h.z) \label{triangle}
 \\
 & \leq c_2 + \ell(h) . \nonumber
\end{align}
Hence there exists  $c\geq0$ such that
\[
 |\tr \chi(h) | \ll \|\chi(h)\| \ll e^{c \ell(h)}.
\]
This proves Proposition~\ref{est_trace}. \qed

\subsection{Proof of Theorem~\ref{thm:selberg_conv}}
For any hyperbolic element $g\in\Gamma$ let $m(g)\in\N$ denote the unique number 
(`multiplicity') such that $g=g_0^{m(g)}$ for a primitive hyperbolic element 
$g_0\in\Gamma$.

By Proposition~\ref{est_trace} we find $c\geq 0$ (which we fix throughout) such that 
uniformly for all hyperbolic elements $h\in\Gamma$ we have
\begin{equation}\label{bound_trace}
 | \tr\chi(h)| \ll N(h)^c 
\end{equation}
and 
\begin{equation}\label{bound_norm}
 \inf_{k\in\Gamma}\|\chi(khk^{-1})\| \ll N(h)^c.
\end{equation}
For $s\in\C$, $\Rea s > c$, the bound~\eqref{bound_norm} yields that for each 
$[g]\in [\Gamma]_\prim$ there exists a representative $g\in\Gamma$ such that
\[
\|\chi(g)N(g)^{-s}\|<1.
\]
It follows that, for all $k\in\N_0$, the equality
\[
 \tr\log\left(1-\chi(g)N(g)^{-(s+k)} \right) = -\sum_{n=1}^\infty 
\frac{N(g^n)^{-(s+k)}}{n}\tr \chi(g^n)
\]
is valid and the right hand side's series is convergent. Therefore, for $\Rea s 
> c$, we have the formal identity
\begin{equation}\label{logszf}
 \log Z(s,\Gamma,\chi) = -\sum_{[g]\in [\Gamma]_\hyp} \frac{1}{m(g)} 
\frac{N(g)^{-s}}{1-N(g)^{-1}} \tr\chi(g)
\end{equation}
and the convergence of the infinite product \eqref{pre_def} is equivalent to the 
convergence of either side  of \eqref{logszf}.

Recall that $\bf 1$ denotes the trivial one-dimensional representation of 
$\Gamma$. Combining \eqref{logszf} and \eqref{bound_trace} we find 
\begin{align*}
|\log Z(s,\Gamma,\chi)| & \leq \sum_{[g]\in [\Gamma]_\hyp} \frac{1}{m(g)} 
\frac{N(g)^{-\Rea s}}{1-N(g)^{-1}} |\tr\chi(g)|
\\
& \ll \sum_{[g]\in [\Gamma]_\hyp} \frac{1}{m(g)} \frac{N(g)^{-(\Rea 
s-c)}}{1-N(g)^{-1}} 
\\
& = |\log Z(\Rea s-c,\Gamma, \bf 1)|.
\end{align*}
The map $s\mapsto\log Z(s-c, \Gamma, \bf 1)$ converges compactly on a right 
half-plane in $\C$ (see Section~\ref{sec:szf}). More precisely, it converges 
compactly on $\{ \Rea s  > \delta+c\}$, where $\delta$ is the Hausdorff 
dimension of the limit set of $\Gamma$. This completes the proof. \qed

\subsection{A side-result}

A variation of the proof of Proposition~\ref{est_trace} shows an interesting 
relation between the norm of $\|\chi(h)\|$ and the displacement $d_\h(z,h.z)$ for 
all hyperbolic elements $h\in\Gamma$. Even though we do not need this result in 
this article, we record it here for purposes of reference. This result is 
crucial for the study of $\chi$-twisted Eisenstein series \cite{FP_Eisenstein}.

\begin{cor}
There exists $c\geq0$ and a continuous map $\Dfunc\colon\h\to\R$ such that uniformly for 
all hyperbolic elements $h\in\Gamma$ and all $z\in\h$, we have
\[
 \| \chi(h)\| \leq \Dfunc(z) e^{c \cdot d_\h(z, h.z)}.
\]
\end{cor}

\begin{proof}
We use the notation from the proof of Proposition~\ref{est_trace}. Without loss 
of generality we suppose that the fundamental domain is chosen such that $i\in 
D$ (we could pick any other generic point $z_0$ and arrange $D$ around $z_0$). 

Let $h\in\Gamma$ be an arbitrary hyperbolic element. We do not suppose that 
$\gamma_h$ intersects the compact core $D_0$. We suppose that $i$ is the point 
to deduce the block decomposition of $h$ as in Section~\ref{block_bound}. In the 
unlikely case that the geodesic connecting $i$ and $h.i$ passes through a vertex 
of a $\Gamma$-translate of $D$, we use a slight perturbation of $i$. Any compact 
perturbation does not change the nature of the results.   

We write $h$ in its block decomposition and proceed to estimate $\|\chi(h)\|$ as 
in the proof of Proposition~\ref{est_trace} until \eqref{est_parab}. In order to 
estimate the $m_j$'s we let $c_h$ be the geodesic segment connecting $i$ and 
$h.i$. Note that we do not request that $i$ is in $\gamma_h(\R)$. This is a 
major difference to the proof of Proposition~\ref{est_trace}. Then 
\eqref{splitpoint} becomes
\[
 d_\h(i,h.i)=\ell(c_h).
\]
We continue to proceed as in the proof of Proposition~\ref{est_trace} until 
\eqref{est_m}, which now reads
\[
 \prod_{j\in J_p} m_j^{d_0-1} \ll e^{4(d_0-1)d_\h(i,h.i)}.
\]
It follows that 
\[
 \|\chi(h)\| \ll e^{c\cdot d_\h(i,h.i)},
\]
where 
\[
 c = c_1 + 4(d_0-1).
\]
Let $z\in\h$. Using \eqref{triangle} shows
\[
 \| \chi(h)\| \ll e^{2d_\h(i,z)} e^{c\cdot d_\h(z,h.z)}. 
\]
This completes the proof.
\end{proof}

\section{Meromorphic continuation of twisted Selberg zeta 
functions}\label{sec:meromorphic}

Throughout this section let $\Gamma$ be a geometrically finite Fuchsian group 
and let $\chi\colon\Gamma\to\GL(V)$ be a finite-dimensional representation with 
non-expanding cusp monodromy. In this section we will show that under the assumption 
that $\Gamma$ admits a so-called strict transfer operator approach (defined in 
Section~\ref{def_stoa} below), the Selberg zeta function $Z_{\Gamma,\chi}$ 
admits a meromorphic continuation to all of $\C$. More precisely, we will show the 
stronger statement that $Z_{\Gamma,\chi}$ equals the Fredholm determinant of a 
family of twisted transfer operators, and that this family admits a meromorphic 
continuation to all of $\C$. In addition, we will provide upper bounds on the order 
of poles in the meromorphic continuation. For precise statements we refer to 
Theorem~\ref{mainthm:strong} below.

In order to define the notion of a strict transfer operator approach we need a 
few preparations.

\subsection{Geometry}\label{sec:geometry}
We consider the Riemann sphere $\wh\C = \C\cup\{\infty\}$ as a complex manifold 
with the standard complex structure which is determined by two charts, the first 
of which is given by the identity on $\C$, and the second is 
\[
\wh\C\smallsetminus\{0\} \to \C,\quad 
z\mapsto 
\begin{cases}
0 & \text{if $z=\infty$}
\\
\frac1z & \text{otherwise.}
\end{cases}
\]
We call a subset of $\wh\C$ \textit{bounded} if its image in at least one of 
these charts is bounded as a subset of $\C$ (endowed with the euclidean metric). 
Note that this notion of boundedness does not coincide with boundedness with 
respect to the chordal metric on $\wh\C$. Note further that a bounded subset of $\wh\C$ is in particular contained in the domain of a chart, and hence $\wh\C$ itself is not a bounded set.

Likewise we consider $\wh\R = \R\cup\{\infty\}$ as a real manifold whose 
manifold structure is determined by the restriction of the two charts for 
$\wh\C$ to $\wh\R$. We call a subset of $\wh\R$ an \textit{interval} if its 
image in at least one of these charts is an interval in $\R$. According to this definition, the empty set and singletons are intervals. The whole manifold $\wh\R$ itself however is not considered to be an interval because it is not contained in the domain of a single chart map.

Let $\wh{\mc V}_f$ denote the set of geodesics on $X\coloneqq \Gamma\backslash \h$ 
which converge to a cusp or funnel in forward time direction. In other words, 
$\wh{\mc V}_f$ is the set of geodesics which, in forward time direction, 
eventually leave any compact set forever. Let $\mc V_f$ denote the (maximal) set 
of geodesics on $\h$ which represent the geodesics in $\wh{\mc V}_f$, and let 
\[
 \bd \coloneqq \bd(\Gamma)\coloneqq \{\gamma(\infty) \mid \gamma\in \mc V_f\} \quad 
\]
denote the `boundary part' of $X$ of the geodesic boundary $\wh\R$ of $\h$. 

We provide a more geometrical description of the set $\bd = \bd(\Gamma)$: 
Suppose that $\mc F$ is a fundamental domain for $\Gamma$ in $\h$, e.\,g., a 
Dirichlet fundamental domain, and consider the closure $\overline{\mc F}$ of 
$\mc F$ in the geodesic closure 
\[
\overline{\h} = \h \cup \wh\R 
\]
of $\h$ (which is identical to the closure of $\h$ in $\wh\C$). Then 
\[
 \bd = \Gamma.\big( \overline{\mc F} \cap \wh\R\big).
\]

For any subset $I$ of $\wh\R$ we set
\[
 I_\st \coloneqq I\smallsetminus\bd.
\]

The subscript `$\st$' abbreviates `strong'. It alludes to the notion of a 
`strong cross section' for the geodesic flow on $\Gamma\backslash\h$. Such cross 
sections are a good source for tuples that satisfy the notion of a 
`strict transfer operator approach' which we will define in the next section (cf.\@~\cite{Pohl_Symdyn2d}).

\subsection{Strict transfer operator approaches}\label{def_stoa}
We say that $\Gamma$ admits a \textit{strict transfer operator approach} if 
there exists a tuple
\begin{equation}\label{tuple}
 \structure\coloneqq\big( A, (I_a)_{a\in A}, (P_{a,b})_{a,b\in A}, (C_{a,b})_{a,b\in 
A}, ((g_p)_{p\in P_{a,b}})_{a,b\in A}\big)
\end{equation}
consisting of
\begin{itemize}
 \item a finite (index) set $A$, 
 \item a family $(I_a)_{a\in A}$ of (not necessarily disjoint) intervals in 
$\wh\R$,
 \item for all $a,b\in A$ a finite (possibly empty) set $P_{a,b}$ of parabolic 
elements in $\Gamma$ and a finite (possibly empty) subset $C_{a,b}$ of $\Gamma$, and
\item for all $a,b\in A$ and all $p\in P_{a,b}$ an element $g_p\in \Gamma$ (which may be the identity)
\end{itemize}
which satisfies Properties~\ref{P1}-\ref{P5} below.

\begin{property}\label{P1}\mbox{ }
\vspace*{-1mm}
\begin{enumerate}[{\rm (i)}]
\item For all $a,b\in A$, all $p\in P_{a,b}$ and $n\in\N$ we have 
$p^{-n}g_p^{-1}.I_{a,\st} \subseteq I_{b,\st}$. Further, $p^n\notin P_{a,b}$ for all $n\geq 2$.
\item For all $a,b\in A$ and all $g\in C_{a,b}$ we have 
$g^{-1}.I_{a,\st}\subseteq I_{b,\st}$.
\item For all $b\in A$ the sets in the family
\[
 \big\{ g^{-1}.I_{a,\st} \ \big\vert\ a\in A,\ g\in C_{a,b}\big\} \cup \big\{ 
p^{-n}g_p^{-1}.I_{a,\st} \ \big\vert\ a\in A,\ p\in P_{a,b},\ n\in\N \big\}
\]
are pairwise disjoint, and
\begin{align*}
 I_{b,\st} &= \bigcup \left\{ g^{-1}.I_{a,\st}\ \left\vert\ a\in A,\ g\in 
C_{a,b}\vphantom{g^{-1}.I_{b,\st} }\right.\right\} 
 \\
 & \ \quad \cup \bigcup \left\{ p^{-n}g_p^{-1}.I_{a,\st}\ \left\vert\ a\in A,\ p\in 
P_{a,b},\ n\in\N \vphantom{p^{-n}.I_{b,\st} }\right.\right\}.
\end{align*}
\end{enumerate}
\end{property}

Property~\ref{P1} induces a discrete dynamical system $(D_\st, F)$, where 
\begin{equation}\label{def_Dst}
 D_\st \coloneqq \bigcup_{a\in A} I_{a,\st}\times \{a\},
\end{equation}
and $F$ is determined by the submaps
\begin{equation}\label{Fmap1}
 g^{-1}.I_{a,\st}\times\{b\} \to I_{a,\st}\times\{a\},\quad (x,b) \mapsto 
(g.x,a)
\end{equation}
and
\begin{equation}\label{Fmap2}
 p^{-n}g_p^{-1}.I_{a,\st}\times\{b\} \to I_{a,\st}\times\{a\},\quad (x,b)\mapsto 
(g_pp^n.x,a)
\end{equation}
for $a,b\in A$, $g\in C_{a,b}$, $p\in P_{a,b}$ and $n\in\N$.

For $n\in\N$ let $\Per_n$ (`periodic') denote the subset of $\Gamma$ which consists of all 
$g\in\Gamma$ for which there exists $a\in A$ such that 
\begin{equation}\label{defPn}  
g^{-1}.I_{a,\st}\times \{a\} \to I_{a,\st}\times\{a\}, \quad (x,a)\mapsto 
(g.x,a)                                                                          
\end{equation}
is a submap of $F^n$.

\begin{property}\label{P2}
The union 
\[
 \Per \coloneqq \bigcup_{n\in\N} \Per_n
\]
is disjoint.
\end{property}

For $n\in\N$ and $h\in \Per_n$ we say that $w(h)\coloneqq n$ is the \textit{word 
length} of $h$. Property~\ref{P2} yields that the word length is well-defined. 

Recall that $[\Gamma]_{\hyp}$ denotes the set of $\Gamma$-conjugacy classes of 
hyperbolic elements in $\Gamma$.

\begin{property}\label{P3}\mbox{ }
\vspace*{-1mm}
\begin{enumerate}[{\rm (i)}]
\item All elements of $\Per$ are hyperbolic.
\item If $h\in \Per$, then its primitive $h_0$ is also contained in $\Per$.
\item For each $[g] \in [\Gamma]_{\hyp}$ there exists a unique $n\in\N$ such 
that there exists $h\in \Per_n$ such that $h$ represents $[g]$. 
\end{enumerate}
\end{property}

Suppose that $[g] \in [\Gamma]_{\hyp}$ is represented by $h$ for some $h\in 
\Per_n$ for some $n\in\N$. Let $m\coloneqq m(h)\in\N$ denote the unique number such that $h=h_0^{m}$ for 
a primitive hyperbolic element $h_0\in\Gamma$, and let $p(h) \coloneqq n/m$. Then we 
set $w(g) \coloneqq w(h)$, $m(g)\coloneqq m(h)$ and $p(g)\coloneqq p(h)$. By 
Property~\ref{P3} these values are well-defined.

\begin{property}\label{P4}
For $[g]\in [\Gamma]_\hyp$, there are exactly $p(g)$ distinct elements $h\in 
\Per_{w(g)}$ such that $h$ represents $[g]$.
\end{property}

\begin{property}\label{P5}
There exists a family $(\mc E_a)_{a\in A}$ of open, bounded, connected, simply 
connected sets in $\wh\C$ such that 
\begin{enumerate}[{\rm (i)}]
\item for all $a\in A$ we have
\[
\overline{I}_{a,\st} \subseteq \mc E_a,
\]
\item\label{P5ii} there exists $\xi\in\PSL_2(\R)$ such that for all $a\in A$ we 
have $\xi.\overline{\mc E}_a \subseteq \C$, and for all $b\in A$ and all $g\in 
C_{a,b}$ we have
\[
g\xi^{-1}.\infty\notin \overline{\mc E}_a, 
\]
\item\label{P5iii} for all $a,b\in A$ and all $g\in C_{a,b}$ we have 
\[
g^{-1}.\overline{\mc E}_a \subseteq \mc E_b,
\]
\item\label{P5iv} for all $a,b\in A$, all $p\in P_{a,b}$ there exists a compact 
subset $K_{a,b,p}$ of $\wh\C$ such that for all $n\in\N$ we have 
\[
p^{-n}g_p^{-1}.\overline{\mc E}_a \subseteq K_{a,b,p} \subseteq \mc E_b,
\]
\item\label{no_fp} for all $a,b\in A$ and all $p\in P_{a,b}$, the set 
$g_p^{-1}.\overline{\mc E}_a$ does not contain the fixed point of $p$.
\end{enumerate}
\end{property}

If the tuple $\structure$ in \eqref{tuple} satisfies 
Properties~\ref{P1}-\ref{P5}, then we call it a \textit{structure tuple} for a 
strict transfer operator approach. The existence of the element $\xi$ in 
Property~\ref{P5}\eqref{P5ii} allows us to find a chart in the manifold atlas of $\wh \C$  which 
contains all the sets $\overline{\mc E}_a$, $a\in A$. Without loss of generality 
we assume throughout that $\xi=\id$. Moreover, to avoid some technical discussions, we assume without loss of generality that for all $a\in A$, 
\[
 \Rea \overline{\mc E}_a \subseteq \overline{\mc E}_a.
\]

As mentioned in the Introduction, the number of geometrically finite Fuchsian 
groups for which strict transfer operator approaches are established is 
increasing \cite{Fried_zetafunctionsI, Artin, Series, Bowen, Bowen_Series, 
Fried_triangle,  Patterson_Perry, Chang_Mayer_extension, 
Manin_Marcolli_transfer, Pohl_Symdyn2d, Moeller_Pohl, Pohl_hecke_infinite, 
Pohl_representation}. The hyperbolic funnel groups are (as Fuchsian Schottky 
groups) among these Fuchsian groups. It is most likely that at least all 
Fuchsian groups with cusps admit such strict transfer operator approaches, and 
that our approach can be extended to non-strict transfer operator approaches 
which then will allow us to cover an even larger class of Fuchsian groups, presumably 
all geometrically finite ones. In Example~\ref{ex:Mayer} below we will provide an explicit example for a strict transfer operator approach for the modular group $\PSL_2(\Z)$. 

For all Fuchsian groups already considered, a strict transfer operator approach can be found for which each of the elements $g_p$ in \eqref{tuple} is the identity. We here allow arbitrary $g_p\in\Gamma$ in order to keep the setting more flexible.

\subsection{Transfer operators and Banach spaces}\label{sec:tos} From now on we 
suppose that $\Gamma$ admits a strict transfer operator approach. Let 
\[
\structure\coloneqq\big( A, (I_a)_{a\in A}, (P_{a,b})_{a,b\in A}, (C_{a,b})_{a,b\in 
A}, ((g_p)_{p\in P_{a,b}})_{a,b\in A} \big)
\]
be a structure tuple, and let $(D_\st, F)$ denote the induced discrete dynamical 
system (see \eqref{def_Dst}-\eqref{Fmap2}). 

As mentioned above, in order to prove the meromorphic continuability of twisted 
Selberg zeta functions, we will use transfer operators to represent these zeta 
functions. In this section we will provide the transfer operators associated to the 
structure tuple and the discrete dynamical system.

For the definition of these transfer operators we will proceed in a two-step process. In the first step we will define 
these operators only formally, considering them as (formal) operators on the space of functions $D_\st\to V$. This way, we have a clear motivation for the defining expressions of the transfer operators. One immediately sees that these operators are not well-defined on the space of all functions $D_\st\to V$ due to convergence problems. One also sees that there is certain freedom for the choice of the function spaces on which the transfer operators act as actual operators. Therefore, in the second step we will define the Banach spaces which we will use as domains for the transfer operators and will show that (and in which way) the transfer operators constitute well-defined operators on these Banach spaces. See Theorem~\ref{mainthm:strong} below.

Let 
\[
 \Fct(D_\st;V) \coloneqq \{ \text{$f\colon D_\st\to V$ function} \}
\]
denote the space of all $V$-valued functions on $D_\st$. Throughout we fix a norm $\|\cdot\|$ on $V$ and an essentially 
sub-multiplicative norm $\|\cdot\|$ on $\GL(V)$. For $s\in\C$, subsets 
$U\subseteq\wh\C$, functions $\psi\colon U\to V$, any $g\in\Gamma$ and any 
$z\in\C$ we set
\[
 \alpha_s\big(g^{-1}\big)\psi(z)\coloneqq \big(g'(z)\big)^s \chi\big(g^{-1}\big) 
\psi(g.z),
\]
whenever this definition makes sense. Throughout we use the standard logarithm (i.\,e., holomorphic on the slit plane $\C\smallsetminus(-\infty,0]$) for the definition of complex powers.

The transfer operator with parameter $s\in\C$ associated to the structure tuple~$\structure$ is formally the operator 
\begin{equation}\label{eq:formalTO}
 \TO_s\colon \Fct(D_\st;V) \to \Fct(D_\st;V)
\end{equation}
given by 
\[
 \big(\TO_s f\big)(z) \coloneqq \sum_{w\in F^{-1}(z)} |F'(w)|^{-s} f(w)
\]
for all $f\in\Fct(D_\st;V)$, $z\in D_\st$. We note that the space $D_\st$ is dense in itself, and hence the derivatives are well-defined. 

We identify $\Fct(D_\st;V)$ with the space
\[
 \left\{ (f_a)_{a\in A}  \ \left\vert\   \text{$f_a\colon I_{a,\st}\to V$ for all $a\in A$} \vphantom{(f_a)_{a\in A}}\right.\right\}
\]
by mapping $f\in \Fct(D_\st;V)$ to the function vector $(f_a)_{a\in A}$ determined by
\[
 \forall\,a\in A\ \forall\, x\in I_{a,\st}\colon f_a(x) = f(x,a).
\]
With respect to this identification, the transfer operator $\TO_s$ is (formally)
given by
\[
 \TO_s \coloneqq \big( \TO_{s,a,b}\big)_{a,b\in A}
\]
where
\begin{equation}\label{TOpart}
 \TO_{s,a,b} \coloneqq \sum_{g\in C_{a,b}} \alpha_s(g) + \sum_{p\in 
P_{a,b}}\sum_{n\in\N} \alpha_s(g_pp^n).
\end{equation}
Thus, if $f\in\Fct(D_\st;V)$ and $\tilde f=\TO_sf$ and $f$ is identified with $(f_a)_{a\in A}$ and $\tilde f$ is identified with $(\tilde f_a)_{a\in A}$, then
\[
 \tilde f_a = \sum_{b\in A} \TO_{s,a,b} f_b
\]
for all $a\in A$. To find the specific expression~\eqref{TOpart} of the `matrix coefficients of the transfer operator' we used the specific form of the submaps of $F$ from \eqref{Fmap1} and \eqref{Fmap2}.

With domain and range as stated in \eqref{eq:formalTO}, the transfer operator $\TO_s$ is not well-defined as the involved infinite sums do not converge for all elements in $\Fct(D_\st;V)$. We now provide a domain on which $\TO_s$ constitutes a well-defined operator (at least for a certain range of $s$; see Theorem~\ref{mainthm:strong} below for the proof). To that end we set 
\begin{align*}
 \mc B(\mc D;V) &\coloneqq \{\text{$\varphi\colon\overline{\mc D} \to V$ continuous, 
$\varphi\vert_{\mc D}$ holomorphic}\}
\end{align*}
for any open subset $\mc D\subseteq \C$. We endow $\mc B(\mc D;V)$ with the supremum norm and its norm topology, which 
turns $\mc B(\mc D;V)$ into a Banach space.

For any family $\mc E_A = (\mc E_a)_{a\in A}$ as provided by Property~\ref{P5} 
we set
\[
 \mc B(\mc E_A;V)\coloneqq \bigoplus_{a\in A} \mc B(\mc E_a;V)
\]
and endow it with the structure of a Banach space via the product space 
structure. 

The results in this paper do not depend on the choice of the family $\mc E_A$. This family serves to provide a thickening of the intervals $\overline I_a$, $a\in A$, into the complex plane such that the relations, in particular the mapping properties, between the intervals are preserved and the operators in \eqref{TOpart} are indeed well-defined when acting on $\mc B(\mc E_A;V)$. To indicate that the essential part of the domain of definitions for the function vectors $(f_a)_{a\in A}$ is the family of intervals $I=(I_a)_{a\in A}$ rather than their thickening, we use the notation 
\begin{equation}\label{def_projBS}
\mc B(I;V)\coloneqq\mc B(\mc E_A;V)
\end{equation}
for this Banach space with the understanding that we may choose any family $\mc E_A$ satisfying Property~\ref{P5} (the same family for all appearances of $\mc B(I;V)$).

For the convenience of the reader we provide an explicit example of a strict transfer operator approach.

\begin{example}\label{ex:Mayer}
Mayer \cite{Mayer_thermo, Mayer_thermoPSL} conducted the seminal transfer operator approach to a Selberg zeta function, namely for the modular surface $\PSL_2(\Z)\backslash\h$. The discrete dynamical system his transfer operator is associated to is the Gauss map, which is intimately related to continued fraction expansions and a discretization of the billiard flow with Neumann boundary value conditions on the triangle surface $\PGL_2(\Z)\backslash\h$. This discretization essentially provides a strict transfer operator approach. However in this setting the discrete group is $\PGL_2(\Z)$, which is not Fuchsian. Thus, strictly speaking, Mayer's transfer operator is not covered by our construction. In turn, this example shows that our construction easily  generalizes to discrete subgroups of $\PGL_2(\R)$ that are not Fuchsian. To avoid additional technicalities we decided to omit this generalization in this article.

Nevertheless, the Gauss map is also closely related to a discretization of the geodesic flow on the modular surface $\PSL_2(\Z)\backslash\h$. The arising discrete dynamical system 
and the associated transfer operators have already been discussed in detail in \cite{Pohl_Symdyn2d, Moeller_Pohl}, for which reason we are here rather brief and only show that the underlying system is a strict transfer operator approach.  We use the notation from the previous sections.

Let $\Gamma = \PSL_2(\Z)$. The boundary part of $\PSL_2(\Z)\backslash\h$ is 
\[
 \bd = \bd(\Gamma)= \Q.
\]
For the underlying system we let the set $A$ consist of two symbols, say
\[
 A\coloneqq \{a,b\}
\]
with $a\not=b$. The intervals are
\[
 I_a \coloneqq (0,1)\quad\text{and}\quad I_b\coloneqq (1,\infty).
\]
The sets of parabolic elements are 
\[
P_{a,a} \coloneqq P_{b,b}\coloneqq \emptyset, \quad P_{a,b} \coloneqq \{p_1\}, \quad P_{b,a} \coloneqq \{p_2\}
\]
with
\[
p_1 := \bmat{1}{-1}{0}{1}, \quad p_2 := \bmat{1}{0}{-1}{1}. 
\]
For all $p\in P_{a,b}\cup P_{b,a}$ we set $g_p\coloneqq\id$. For all $n\in\N$ we have
\begin{align*}
p_1^{-n}.I_{a, \st}  &= (n,n+1)_\st \quad \subseteq I_{b,\st}, \\
p_2^{-n}.I_{b, \st} &= \left( \frac{1}{n+1}, \frac{1}{n}\right)_\st \quad\subseteq I_{a,\st}.
\end{align*}
We refer to \cite{Moeller_Pohl} for the proof that Properties~\ref{P1}-\ref{P5} are satisfied.

The induced discrete dynamical system is 
\begin{align*}
p_1^{-n}.I_{a,\st} \times \{b\} &\to I_{a,\st} \times \{a\},\\
 (x,b) &\mapsto (p_1^n.x, a) = ( x-n, a),\\
p_2^{-n}.I_{b,\st} \times \{ a \} & \to I_{b, \st} \times \{ b\},\\
(x,a) &\mapsto (p_2^n.x, b) = \left(\frac{x}{-nx+1},b\right).
\end{align*}
The transfer operator is
\begin{equation}\label{modular_TO}
\TO_s = \mat{0}{\sum\limits_{n \in \N} \alpha_s(p_1^n)}{
\sum\limits_{n \in \N} \alpha_s(p_2^n)}{0}
\end{equation}
acting on functions (with sufficient regularity) of the form
\[
f = \begin{pmatrix} 
f_a\colon I_{a,\st} \to V \\
f_b\colon I_{b, \st} \to V 
\end{pmatrix}.
\]
The relation of the transfer operator $\TO_s$ to Mayer's transfer operator $\TO_s^M$ is given by diagonalization, and taking the representation $\chi$ to be trivial and one-dimensional. Changing the basis of the function space shows that $\TO_s$ is conjugate to
\[
 \mat{\TO_s^M}{0}{0}{-\TO_s^M}.
\]
Geometrically, this conjugation means that the geodesic flow on $\PSL_2(\Z)\backslash\h$ is split into two copies of the billiard flow on $\PGL_2(\Z)\backslash\h$, one with Neumann boundary value conditions and one with Dirichlet boundary value conditions. We refer to \cite{Pohl_spectral_hecke} for a more detailed discussion.

\end{example}

\subsection{Statement of Main Theorem for meromorphic continuability (strong 
version)}\label{sec:statement}

With the preparations from Sections~\ref{sec:geometry}-\ref{sec:tos} we are now 
able to state the strong version of the main theorem regarding the meromorphic 
continuability of twisted Selberg zeta functions.

For any Banach space $\mc B$, we let 
\[
 \Op(\mc B)\coloneqq \{ \text{$L\colon \mc B\to \mc B$ linear, bounded}\}
\]
denote the bounded linear operators on $\mc B$.

\begin{thm}\label{mainthm:strong}
Let $\Gamma$ be a geometrically finite Fuchsian group which admits a strict 
transfer operator approach, and let $\chi\colon \Gamma\to \GL(V)$ be a 
finite-dimensional representation of $\Gamma$ with non-expanding cusp monodromy. 
Let 
\[
 \mc S \coloneqq \big( A, (I_a)_{a\in A}, (P_{a,b})_{a,b\in A}, (C_{a,b})_{a,b\in A}, ((g_p)_{p\in P_{a,b}})_{a,b\in A} 
\big)
\]
be a structure tuple as in \eqref{tuple}, and let $\mc B(I;V)$ denote the 
associated Banach space as in \eqref{def_projBS}. For $s\in\C$ let $\TO_s$ 
denote the associated (formal) transfer operator. Further, for $p\in\Gamma$ 
parabolic let $\jc(p)$ denote the family (with multiplicities) of the lengths of 
Jordan chains for eigenvalue $1$ of $\chi(p)$, and let 
\[
 d_1 \coloneqq \max\{ \jc(p) \mid \text{$p\in\Gamma$ parabolic} \}
\]
denote the maximal length of a Jordan chain ranging over all $\chi(p)$, $p\in\Gamma$ 
parabolic, with eigenvalue $1$. Further, let $d_0$ denote the 
maximal length of a Jordan chain ranging over all $\chi(p)$, $p\in\Gamma$ 
parabolic, and all eigenvalues. 

Then the following properties hold:
\begin{enumerate}[{\rm (i)}]
\item\label{strongi} For $\Rea s > \frac{d_0}{2}$, the (formal) transfer 
operator $\TO_s$ defines an (actual) operator on $\mc B(I;V)$. It is bounded and nuclear of 
order $0$.
\item\label{strongii} The map 
\[
 \left\{ s\in\C \ \left\vert\  \Rea s > \tfrac{d_0}2\right.\right\} \to 
\Op\big(\mc B(I;V)\big),\quad s\mapsto \TO_s
\]
extends meromorphically to all of $\C$ with values in nuclear operators of order 
$0$ in $\mc B(I;V)$. All poles are simple and of the form $\tfrac12(d_1-k)$ with 
$k\in\N_0$. For each pole $s_0$, there exists a neighborhood $U$ of $s_0$ such 
that the meromorphic continuation $\wt\TO_s$ is of the form
\[
 \wt\TO_s = \frac{1}{s-s_0}\mc A_s + \mc B_s,
\]
where the operators $\mc A_s$ and $\mc B_s$ are holomorphic on $U$, and $\mc 
A_s$ is of rank at most 
\[
 \sum_{a,b\in A} \sum_{p\in P_{a,b}} \sum_{d\in \jc(p)} {d+1 \choose 2}.
\]
\item\label{strongiii} For $\Rea s \gg 1$, $Z_{\Gamma,\chi}(s) = \det(1-\TO_s)$.
\item\label{strongiv} The Selberg zeta function $Z_{\Gamma,\chi}$ extends 
meromorphically to all of $\C$. Its poles are contained in $\tfrac12(d_1-\N_0)$. 
The order of any pole is at most
\[
 \sum_{a,b\in A} \sum_{p\in P_{a,b}} \sum_{d \in \jc(p)} {d+1 \choose 2}.
\]
\end{enumerate}
\end{thm}

For unitary representations $\chi$, the proof of Theorem~\ref{mainthm:strong} is 
by now standard \cite{Ruelle_zeta, Fried_zetafunctionsI, Mayer_thermo, 
Pohl_representation}. Thus, our main emphasis is on discussing the necessary 
changes and extensions of the standard proof in order to accommodate arbitrary 
representations with non-expanding cusp monodromy. 

The proof of Theorem~\ref{mainthm:strong} is split into several parts: In 
Section~\ref{sec:nuclear} below we will show 
Theorem~\ref{mainthm:strong}\eqref{strongi}. Compared to the standard proof the 
essential part to show for Theorem~\ref{mainthm:strong}\eqref{strongi} is the 
convergence of the infinite sums featuring in the transfer operator, see 
\eqref{TOpart}. Every such infinite sum is governed by a parabolic element. 
Hence, \textit{a priori}, the growth of $\chi(p^n)$ as $n\to\infty$ might cause 
divergence. However, the logarithm-type law Lemma~\ref{jordanblock} gives us 
sufficient control of this growth in order to establish convergence.

In Section~\ref{sec:fredholm} below we will prove 
Theorem~\ref{mainthm:strong}\eqref{strongiii}. The key result for this step is an easy-to-evaluate
formula for the trace of $\alpha_s(h)$, $h\in\Gamma$ hyperbolic, which we will
provide in Lemma~\ref{lem:formulatrace} below.

In Section~\ref{sec:merom} below we will show 
Theorem~\ref{mainthm:strong}\eqref{strongii}. This result is based on a 
reduction to the Lerch transcendent and its derivatives, which allows us to 
handle the contribution of non-diagonalizable families of endomorphisms 
$(\chi(p^n))_{n\in\N}$ for $p\in\Gamma$ parabolic. 

We remark that for some Fuchsian groups, in particular for Fuchsian Schottky 
groups, the standard setup uses Hilbert spaces of holomorphic $L^2$-functions in 
place of $\mc B(\mc E_A;V)$ \cite{Guillope_Lin_Zworski}. The proof of 
Theorem~\ref{mainthm:strong} applies, \textit{mutatis mutandis}, to this 
situation. We will not discuss it separately.  

Throughout Sections~\ref{sec:nuclear}-\ref{sec:proofcomplete} we will use, without 
any further reference, the notation from 
Sections~\ref{sec:geometry}-\ref{sec:statement}, in particular regarding the 
structure tuple $\mc S$ and the transfer operators $\TO_s$. Further we fix a 
family $\mc E_A = (\mc E_a)_{a\in A}$ as provided by Property~\ref{P5}.

\subsection{Nuclearity of transfer operators}\label{sec:nuclear}
In Proposition~\ref{nuclear} below we will show that $\TO_s$ defines a nuclear 
operator of order $0$ on $\mc B(\mc E_A;V)$, which then proves Theorem~\ref{mainthm:strong}\eqref{strongi}. 

Essential for the proof of Proposition~\ref{nuclear} is to understand the 
influence of the families $(\chi(p^n))_{n\in\N}$, $p\in\Gamma$ parabolic, on the 
growth of the transfer operators $\TO_s$. For this, we make crucial use of 
Lemma~\ref{jordanblock}.

\begin{prop}\label{nuclear}
Let $s\in\C$, $\Rea s > \tfrac{d_0}2$. Then the transfer operator $\TO_s$ 
defines a bounded self-map of $\mc B(\mc E_A;V)$. As such $\TO_s$ is nuclear of order 
$0$.
\end{prop}

\begin{proof}
For any open subset $\mc D\subseteq \C$ we let
\begin{align*}
 \mc H(\mc D;V) &\coloneqq \{ \text{$f\colon \mc D\to V$ holomorphic} \}
\end{align*}
denote the space of holomorphic functions $\mc D\to V$ and endow it with the 
compact-open topology, which turns $\mc H(\mc D;V)$ into a nuclear space 
\cite[I, \S2, no.\@ 3, Corollary]{Grothendieck_produit}. We further set 
\[
 \mc H(\mc E_A;V) \coloneqq \bigoplus_{a\in A} \mc H(\mc E_a;V),
\]
and endow it with the product space structure. Hence, $\mc H(\mc E_A;V)$ is a 
nuclear space.

We start by showing that the transfer operator $\TO_s$ defines a map from $\mc 
H(\mc E_A;V)$ to $\mc B(\mc E_A;V)$. To that end let 
\[
f=\bigoplus_{a\in A} f_a\in \mc H(\mc E_A;V)
\]
and set 
\[
h\coloneqq \TO_s f, \quad h=\bigoplus_{a\in A} h_a.
\]
Let $a\in A$. Then 
\begin{align}
 h_a &= \sum_{b\in A} \TO_{s,a,b}f_b \nonumber
  \\
 & =  \sum_{b\in A} \left( \sum_{g\in C_{a,b}} \alpha_s(g) f_b + \sum_{p\in 
P_{a,b}}\sum_{n\in\N} \alpha_s(g_pp^n) f_b\right). \label{TO_partial}
\end{align}
We claim that $h_a\in \mc B(\mc E_a;V)$. Let $b\in A$. 
Property~\ref{P5}\eqref{P5ii}-\eqref{P5iii} immediately imply that for all $g\in C_{a,b}$,  
the map
\begin{equation}\label{unibound_single}
 z\mapsto \alpha_s(g)f_b(z) = \big((g^{-1})'(z)\big)^s \chi(g) f_b(g^{-1}.z)
\end{equation}
is well-defined on $\overline{\mc E}_a$, continuous on $\overline{\mc E}_a$ and 
holomorphic on $\mc E_a$. This and the finiteness of $C_{a,b}$ implies that the 
first summand in \eqref{TO_partial} defines an element of $\mc B(\mc E_a;V)$. 
Therefore it remains to show that for each $p\in P_{a,b}$,  the infinite sum
\[
 F\coloneqq F_{s,p,b}\coloneqq\sum_{n\in\N}\alpha_s(g_pp^n)f_b
\]
defines an element of $\mc B(\mc E_a;V)$. Recall that $d_0$ denotes the maximal 
length of a Jordan chain of $\chi(q)$, where $q$ ranges over all parabolic 
elements in $\Gamma$. Lemma~\ref{jordanblock} shows 
\[
 \| \chi(p^{n})\| \ll n^{d_0-1}
\]
for all $n\in\N$. Further, by Property~\ref{P5}\eqref{P5iv}, there exists a 
compact subset $K_{a,b,p}$ of $\wh\C$ such that for each $n\in\N$,  
\begin{equation}\label{recallP5}
p^{-n}g_p^{-1}.\overline{\mc E}_a \subseteq K_{a,b,p} \subseteq \mc E_b.
\end{equation}
Since $f_b$ is continuous on $\mc E_b$, $\|f_b(p^{-n}g_p^{-1}.z)\|$ is bounded on 
$\overline{\mc E_a}$, uniformly in $n\in\N$. Thus,
\[
 \sup\Big\{ \left\|\chi(g_pp^{n}) f_b(p^{-n}g_p^{-1}.z) \right\| \colon z\in \overline{\mc 
E}_a,\ n\in\N\Big\} \ll n^{d_0-1}.
\]
Since $p$ is parabolic, and hence conjugate to $\textbmat{1}{1}{0}{1}$ in 
$\PSL_2(\R)$, there exist $c,d\in\R$ such that
\begin{equation}\label{formp}
 p^n = \bmat{1+n cd}{n d^2}{-n c^2}{1-n cd} 
\end{equation}
for all $n\in\Z$. The boundedness of $K_{a,b,p}$ in \eqref{recallP5} implies that $c\not=0$.

Let 
\[
 g_p = \bmat{g_{11}}{g_{12}}{g_{21}}{g_{22}}.
\]
Then 
\[
 p^{-n}g_p^{-1} = \bmat{*}{*}{g_{22}nc^2 - g_{21}(1+ncd)}{-g_{12}nc^2 + g_{11}(1+ncd)}.
\]

By Property~\ref{P5}\eqref{P5iv} or, equivalently, by \eqref{recallP5}, 
\[
 p^{-n}g_p^{-1}.z\not=\infty
\]
for all $z\in\overline{\mc E}_a$, $n\in\N$, and hence
\[
 \Rea\left( \big(g_{22}nc^2-g_{21}(1+ncd)\big)z - g_{12}nc^2+g_{11}(1+ncd) \right) \not=0
\]
Thus, for all $z\in\overline{\mc E}_a$, $n\in\N$,
\[
\big(\big(p^{-n}g_p^{-1}\big)'(z)\big)^s = \left( \left( nc\big( c(g_{22}z-g_{12}) + d(-g_{21}z+g_{11})\big) - g_{21}z+g_{11} \right)^{-2} \right)^s
\]
is well-defined, and 
\[
 \left| \left(\big(p^{-n}g_p^{-1}\big)'(z)\right)^s \right| \ll_s |nc\big( c(g_{22}z-g_{12}) + d(-g_{21}z+g_{11})\big) - g_{21}z+g_{11}|^{-2\Rea s}.
\]
By Property~\ref{P5}\eqref{no_fp}, 
\[
 c(g_{22}z-g_{12}) + d(-g_{21}z+g_{11})\not=0
\]
for all $z\in\overline{\mc E}_a$. In combination with $c\not=0$ (see above) it follows that 
\[
 \sup_{z\in\overline{\mc E}_a}\left| \left(\big(p^{-n}g_p^{-1}\big)'(z)\right)^s \right| 
\ll_s n^{-2\Rea s}.
\]
Therefore
\begin{equation}\label{unibound_parabolic}
 \sup_{z\in\overline{\mc E}_a} \left\| \alpha_s(g_pp^n) f_b(z) \right\| \ll_s 
n^{d_0-1-2\Rea s}.
\end{equation}
This implies that for $\Rea s > \tfrac{d_0}2$, the sum $F=\sum_{n\in\N} 
\alpha_s(g_pp^n)f_b$ converges uniformly on $\overline{\mc E}_a$. By Weierstrass' 
Theorem, $F$ is continuous on $\overline{\mc E}_a$ and holomorphic on $\mc E_a$. 
Therefore $F\in \mc B(\mc E_a;V)$, and hence $\TO_s\colon \mc H(\mc E_A;V) \to 
\mc B(\mc E_A;V)$.

As in the standard proof \cite{Ruelle_zeta, Fried_zetafunctionsI, Mayer_thermo, 
Pohl_representation}, one can now deduce that, as a map from $\mc H(\mc 
E_A;V)$ to $\mc B(\mc E_A;V)$, $\TO_s$ is nuclear of order $0$. Then one further 
concludes that $\TO_s$ remains nuclear of order $0$ when considered as a map 
$\mc B(\mc E_A;V)\to \mc B(\mc E_A;V)$.
\end{proof}

\subsection{The twisted Selberg zeta function as a Fredholm 
determinant}\label{sec:fredholm}

In this section we will show that, for $\Rea s \gg 1$, the twisted Selberg zeta function~$Z_{\Gamma,\chi}$ equals the 
Fredholm determinant of the transfer operator family $\TO_s$. This immediately 
proves Theorem~\ref{mainthm:strong}\eqref{strongiii}.

In order to simplify the proof of the trace formula in 
Lemma~\ref{lem:formulatrace} below, we set
\begin{equation}\label{tauC}
 \tau_s^\C\big(h^{-1}\big)\varphi(x) \coloneqq \big( h'(x)\big)^s\varphi(h.x)
\end{equation}
for subsets $U\subseteq\R$, functions $\varphi\colon U\to\C$, $h\in\Gamma$ and 
$x\in U$, whenever it makes sense.

\begin{lemma}\label{lem:formulatrace}
Let $\mc D$ be an open, connected, simply connected, bounded 
subset of~$\wh\C$ and suppose that $h\in \Gamma$ satisfies $h^{-1}.\overline{\mc D} \subseteq 
\mc D$. Then the map
\[
 \alpha_s(h) \colon \mc B(\mc D; V) \to \mc B(\mc D; V)
\]
is nuclear of order $0$  with trace
\[
 \Tr \alpha_s(h) = \frac{N(h)^{-s}}{1-N(h)^{-1}} \tr \chi(h).
\]
\end{lemma}

\begin{proof}
Note that the element $h$ necessarily is hyperbolic. Thus, the norm $N(h)$ is 
indeed defined.

If $(V,\chi)$ is the trivial one-dimensional representation, then $\alpha_s(h) = 
\tau_s^\C(h)$. Note that $N(h) = N(h^{-1})$. By \cite{Ruelle_zeta} we have
\begin{equation}\label{trace1d}
 \Tr \tau_s^\C(h) = \frac{N(h)^{-s}}{1-N(h)^{-1}},
\end{equation}
which shows the statement of the Lemma for this particular case. 

Let $(V,\chi)$ now be a generic representation with non-expanding cusp monodromy 
 and fix a basis $B_V\coloneqq \{v_1,\ldots, v_d\}$ of $V$. The map
\begin{equation}\label{traceiso}
 \mc B(\mc D;\C)^d \to \mc B(\mc D;V),\quad (f_j)_{j=1}^d \mapsto \sum_{j=1}^d 
f_jv_j
\end{equation}
is an isomorphism of Banach spaces. Let 
\[
 M = (m_{ij})_{i,j=1,\ldots, d}
\]
be the $d\times d$-matrix that represents $\chi(h)$ with respect to $B_V$. Under 
the isomorphism~\eqref{traceiso}, the action of $\alpha_s(h)$ is then given by
\[
 \left( m_{ij}\tau_s^\C(h) \right)_{i,j=1,\ldots, d}.
\]
Since 
\[
 \mc B(\mc D;\C)^d \cong \bigoplus_{j=1}^d \mc B(\mc D;\C)v_j
\]
is a topological direct sum, \cite[Chapitre~II.3, 
Prop.~2]{Grothendieck_fredholm} shows that  
\begin{align*}
\Tr \alpha_s(h) & = \sum_{j=1}^d m_{jj} \Tr \tau_s^\C(h).
\end{align*}
Combining this equality with \eqref{trace1d} completes the proof.
\end{proof}

\begin{prop}\label{selbergTO}
For $\Rea s \gg 1$, $Z_{\Gamma,\chi}(s) = \det(1-\TO_s)$.
\end{prop}

\begin{proof}
Recall the constant $C=C(\Gamma,\chi)$ from Theorem~\ref{thm:selberg_conv}, and 
that $d_0$ denotes the maximal length of a Jordan block of $\chi(p)$, 
$p\in\Gamma$ parabolic. Let $s\in\C$ with $\Rea s> \max\{C, d_0/2\}$.  By 
Theorem~\ref{thm:selberg_conv}, $Z_{\Gamma,\chi}(s)$ converges. Further 
Proposition~\ref{nuclear} implies that the Fredholm determinant of $\TO_s$ 
exists. Analogously to the case of the trivial one-dimensional representation 
one sees that 
\[
 \det\left( 1 - \TO_s \right) = \exp\left( - \sum_{n\in\N} \frac1n \Tr 
\TO_s^n\right).
\]
Property~\ref{P5} combined with \eqref{defPn} shows that for $n\in\N$, $h\in 
\Per_n$ and $a\in A$ as in \eqref{defPn} we have
\[
 h^{-1}.\overline{\mc E}_a \subseteq \mc E_a. 
\]
By \cite[Chapitre II.3, Prop.~2]{Grothendieck_fredholm} and 
Lemma~\ref{lem:formulatrace}, 
\[
 \Tr \TO_s^n = \sum_{h\in \Per_n} \Tr \alpha_s(h) = \sum_{h\in \Per_n} 
\frac{N(h)^{-s}}{1-N(h)^{-1}} \tr \chi(h).
\]
Then Properties~\ref{P3} and \ref{P4} imply 
\begin{align}
\log Z_{\Gamma,\chi}(s) & = \sum_{[g]\in [\Gamma]_\prim} \sum_{k=0}^\infty 
\log\left( \det\left(1-\chi(g)N(g)^{-(s+k)} \right) \right) \nonumber
\\
& = -\sum_{[g]\in [\Gamma]_\prim} \sum_{m=1}^\infty \frac1m 
\frac{N(g^m)^{-s}}{1-N(g^m)^{-1}} \tr \chi(g^m) \nonumber
\\
& = -\sum_{[g]\in [\Gamma]_\hyp} \frac{1}{m(g)} \frac{N(g)^{-s}}{1-N(g)^{-1}}\tr 
\chi(g) \nonumber
\\
& = -\sum_{[g]\in [\Gamma]_\hyp} \frac{p(g)}{w(g)} \frac{N(g)^{-s}}{1-N(g)^{-1}}\tr 
\chi(g) \nonumber
\\
& = -\sum_{w=1}^\infty\frac1w \sum_{h\in \Per_w} \frac{N(h)^{-s}}{1-N(h)^{-1}} 
\tr\chi(h) \nonumber
\\
& = -\sum_{w=1}^\infty \frac1w \Tr \TO_s^w \nonumber
\\
& = \log\left( \det\left(1-\TO_s\right)\right). \nonumber
\end{align}
This completes the proof.
\end{proof}

\subsection{Meromorphic continuation of the transfer operator 
family}\label{sec:merom}

In this section we will show that the map
\[
 \left\{ s\in\C \ \left\vert\ \Rea s > \tfrac{d_0}{2} \right.\right\} \to 
\Op\big( \mc B(\mc E_A;V)\big), \quad s\mapsto \TO_s,
\]
extends meromorphically to all of $\C$ with values in nuclear operators of order 
$0$ on $\mc B(\mc E_A;V)$. This means that there exists a discrete set 
$P\subseteq\C$ (`candidates for poles') and for each $s\in\C\smallsetminus P$ we 
find an operator $\wt\TO_s\colon \mc B(\mc E_A;V) \to \mc B(\mc E_A;V)$ which is 
nuclear of order $0$ and which equals $\TO_s$ for $\Rea s > \tfrac{d_0}2$. 
Moreover, for each $f\in\mc B(\mc E_A;V)$ and each point 
\[
z\in\mc E\coloneqq \bigcup_{a\in A}\mc E_a
\]
the map $s\mapsto \wt\TO_sf(z)$ is meromorphic with poles contained in $P$, and 
the map $(s,z)\mapsto \wt\TO_sf(z)$ is continuous on $(\C\smallsetminus P)\times \mc 
E$. 

If $\chi$ is the trivial one-dimensional representation, then the Hurwitz zeta 
function plays a crucial role in the proof of this meromorphic continuability. 
For unitary representations $\chi$, the Lerch zeta function features. For generic 
representations $\chi$ with non-expanding cusp monodromy, we will take advantage of 
the Lerch transcendent and its derivatives. The main emphasis of this section is 
to show how the Lerch transcendent allows us to generalize the standard proof 
for trivial and unitary representations to the larger class of representations 
with non-expanding cusp monodromy. This result shows 
Theorem~\ref{mainthm:strong}\eqref{strongii}.

We denote the \textit{Lerch transcendent} by
\[
 \Phi(s,\lambda,w) \coloneqq \sum_{n=0}^\infty \frac{\lambda^n}{(n+w)^s}.
\]
For $m\in\N_0$ we set
\[
 \Phi_m(s,\lambda,w) \coloneqq \frac{1}{m!} 
\frac{\partial^m\Phi}{\partial\lambda^m}(s,\lambda,w) = \sum_{n=m}^\infty 
{n\choose m} \frac{\lambda^{n-m}}{(n+w)^s}.
\]
For $|\lambda|=1$ and $w\notin -\N_0$ (the values we need) the series converges 
and defines a holomorphic function on $\{s\in\C\mid\Rea s > m+1\}$. For 
$\lambda\not=1$ it extends in $s$ to an entire function. For $\lambda=1$ it 
extends meromorphically in $s$ to all of $\C$ with all poles contained in $\{1,\ldots, 
m+1\}$, all simple, and $m+1$ being a pole. In both cases, the meromorphic extension is continuous on its domain of 
holomorphy as a function in the two variables $(s,w)$. To be more precise, the extension is unique for $w\in\C\smallsetminus (-\infty,0]$. The additional extension in the $w$-variable requires a choice of an extension of the complex logarithm. 

For $\Phi=\Phi_0$ proofs of these statements can be found in the literature, for $\Phi_m$ with $m>0$ they can be deduced from the one for $\Phi_0$ or shown in analogy. For the convenience of the reader, we provide a proof in Section~\ref{sec:lerch} below. We refer to 
\cite{Lagarias_Li} for detailed discussions on the meromorphic continuability 
of the Lerch transcendent as a function of all three variables.

\begin{prop}\label{mainprop}
The map $s\to \TO_s$ extends to a meromorphic function on $\C$ with values in 
nuclear operators of order $0$ on $\mc B(\mc E_A;V)$. The possible poles are all 
simple and contained in $\tfrac12(d_1-\N_0)$ (recall $d_1$ from Theorem~\ref{mainthm:strong}). For each pole $s_0$, there is a 
neighborhood $U$ of $s_0$ such that the meromorphic continuation $\wt\TO_s$ is 
of the form
\[
 \wt\TO_s = \frac{1}{s-s_0}\mc A_s + \mc B_s,
\]
where the operators $\mc A_s$ and $\mc B_s$ are holomorphic on $U$, and $\mc 
A_s$ is of rank at most
\[
 \sum_{a,b\in A} \sum_{p\in P_{a,b}} \sum_{d \in \jc(p)} {d+1 \choose 2}.
\]
\end{prop}

\begin{proof}
For Banach spaces $\mc B_1$, $\mc B_2$ we denote by $\Op(\mc B_1;\mc B_2)$ the 
space of all linear operators $\mc B_1\to\mc B_2$. 

It suffices to show that for $a,b\in A$ and $p\in P_{a,b}$ the map
\[
 \{s\in\C\mid\Rea s > d_0/2\} \to \Op\big(\mc B(\mc E_b;V); \mc B(\mc 
E_a;V)\big),\quad s\mapsto  L_s\coloneqq\sum_{n\in\N} \alpha_s(g_pp^n) 
\]
extends meromorphically to all of $\C$ with values in nuclear operators of order 
$0$ and all poles contained in $P$. Since
\[
 \sum_{n\in\N} \alpha_s(g_pp^n) = \alpha_s(g_p) \sum_{n\in\N}\alpha_s(p^n),
\]
it also suffices to show that 
\[
 \{s\in\C\mid\Rea s > d_0/2\} \to \Op\big(\mc B(\mc E_b;V);\alpha_s(g_p^{-1})\mc B(\mc 
E_a;V)\big),\quad s\mapsto  \sum_{n\in\N} \alpha_s(p^n)
\]
extends meromorphically to all of $\C$. Hence, without loss of generality, we may assume that $g_p=\id$. 

Let $a,b\in A$ and $p\in P_{a,b}$. Represent $\chi(p)$ in Jordan 
normal form  and fix a Jordan block. Let $\lambda\in\C$ be the eigenvalue of 
this Jordan block, and let $d_b$ denote the size of the Jordan block (that is, 
the length of the defining Jordan chain). It suffices to study the meromorphic 
continuability of the maps
\begin{align*}
 \left\{s\in\C\ \left\vert\ \Rea s > \tfrac{d_0}2 \right.\right\} & \to 
\Op\big(\mc B(\mc E_b;\C); \mc B(\mc E_a;\C)\big)
\\
s & \mapsto H_{s,m} \coloneqq \sum_{n\in\N} {n\choose m} \lambda^{n-m} \tau_s^\C 
(p^n)
\end{align*}
for $m\in \{0,1,\ldots, d_b-1\}$. Recall that by hypothesis we have 
$|\lambda|=1$.

Recall that (without loss of generality) $\mc E_a$ and $\mc E_b$ are subsets of 
$\C$. For $M\in\N_0$, let 
\[
Q_M, P_M\colon \mc B(\mc E_b;\C) \to \mc B(\mc E_b;\C)
\]
be the operator that maps a function to its Taylor polynomial of degree $M$ 
centered at the fixed point $x_p$ of $p$ or subtracts this Taylor 
polynomial from it, respectively. Thus we set
\[
 Q_M(f)(z) \coloneqq \sum_{k=0}^M \frac{f^{(k)}(x_p)}{k!} (z-x_p)^k
\]
and
\[
 P_M(f)(z) \coloneqq (1-Q_M)f(z) \coloneqq f(z) - \sum_{k=0}^M \frac{f^{(k)}(x_p)}{k!} 
(z-x_p)^k.
\]
Then
\begin{align*}
 H_{s,m} = H_{s,m}\circ Q_M + H_{s,m}\circ P_M.
\end{align*}
We first study $H_{s,m}\circ Q_M$. By Property~\ref{P5}\eqref{P5iv}, 
$\overline{\mc E_a}$ does not contain the fixed point $x_p$ of $p$. Using the 
explicit form \eqref{formp} for $p^n$, $n\in\N_0$, we find $x_p=-d/c$. Thus, for 
$z\in\mc E_a$ we have $cz+d\not=0$ and 
\[
 p^{-n}.z - x_p = \left( c^2 \left( n + \frac{1}{c ( cz+d )} \right) 
\right)^{-1}.
\]
By Property~\ref{P5}\eqref{P5iv}, 
\[
\Rea\big( n c ( cz+d ) + 1\big) \not= 0
\]
for all $z\in\overline{\mc E}_a$, $n\in\N$. Thus, 
\begin{align*}
 \left( \left(p^{-n}\right)'(z) \right)^s &=\left( \big( n c ( cz+d ) + 
1\big)^{-2}\right)^s 
 \\
 & = e^{-2si\cdot \varphi(n,z)} \left(\left( c (cz+d) \right)^{-2}\right)^s \left(\left( n + 
\frac{1}{ c (cz+d) }\right)^{-2}\right)^s
\end{align*}
for a function $\varphi\colon\N\times\overline{\mc E}_a\to 2\pi\Z$. The function 
$\varphi$ is continuous and hence constant in the $z$-variable. Further, since 
the argument (or phase) of  $n + \frac{1}{ c (cz+d) }$ tends to $0$ as 
$n\to\infty$, there exists $n_0\in\N$ such that $\varphi$ is constant on 
$\N_{\geq n_0}\times\overline{\mc E}_a$ (note that $\frac{1}{ c (cz+d) }$ is 
bounded on $\overline{\mc E}_a$). Without loss of generality we may assume that 
$\varphi = 0$, and that $\Rea\big( c(cz+d) \big)>0$ for all $z\in\overline{\mc E}_a$. Hence,
\[
 \left( \left(p^{-n}\right)'(z) \right)^s = \left( c (cz+d) \right)^{-2s} \left( 
n + \frac{1}{ c (cz+d) }\right)^{-2s}.
\]
Then 
\begin{align*}
(H_{s,m}&\circ Q_M)f(z) = \sum_{k=0}^M \frac{f^{(k)}(x_p)}{k!} \sum_{n\in\N} 
{n\choose m} \lambda^{n-m} \left( \big(p^{-n}\big)'(z)\right)^s \big( p^{-n}.z - 
x_p\big)^k
\\
& = \sum_{k=0}^M \frac{f^{(k)}(x_p)}{k!} c^{-2} \big(c (cz+d) \big)^{-2s} 
\sum_{n\in\N} {n\choose m} \lambda^{n-m} \left( n + \frac{1}{c (cz+d) 
}\right)^{-2s-k}
\\
& = \sum_{k=0}^M \frac{f^{(k)}(x_p)}{k!} c^{-2} \big( c (cz+d) \big)^{-2s} 
\Phi_m\left(2s+k, \lambda, \delta_{0,m} + \frac{1}{ c (cz+d) }\right),
\end{align*}
where 
\[
\delta_{0,m} \coloneqq \begin{cases} 1 & \text{if $m=0$} \\ 0 & \text{otherwise.} 
\end{cases}
\]
The properties of $\Phi_m$, as recalled just before the statement of this proposition, yield the existence of a meromorphic continuation of $s\mapsto 
(H_{s,m}\circ Q_M)f(z)$ with the localization and order of poles as claimed. Note that $\Phi_m$ may only contribute a pole if $\lambda=1$, and hence the first possible pole is located at $s=d_1/2$ instead of $d_0/2$.

The extension of $H_{s,m}\circ Q_M$ is nuclear of order $0$ as a finite rank 
operator.

We now study $H_{s,m}\circ P_M$. We fix an open ball $B$ in $\C$ around $x_p$ 
such that $B\subseteq \mc E_b$. Let $f\in \mc B(\mc E_b)$ and set $f_M\coloneqq 
P_M(f)$. For all $z\in B$ we have
\[
 |f_M(z)| \ll |z-x_p|^{M+1}.
\]
Further, for all $z\in \C$, we have 
\[
 \lim_{n\to\infty} p^{-n}.z = x_p.
\]
Let $z\in \mc E_a$. Then there exists $n_0=n_0(z)\in\N$ such that 
\[
 p^{-n}.z \in B
\]
for $n>n_0$. Thus,
\begin{align*}
\left| \left(H_{s,m}\circ P_M\right) f (z) \right| &= \left| \sum_{n\in\N} 
{n\choose m} \lambda^{n-m} \tau_s^\C(p^n) f_M(z)\right|
\\
& \leq \left| \sum_{n\leq n_0} {n\choose m} \lambda^{n-m} 
\tau_s^\C(p^n)f_M(z)\right|
\\
& \quad + \sum_{n>n_0} {n\choose m} \left| \big(p^{-n}\big)'(z)\right|^{\Rea s} 
\left| f_M(p^{-n}.z) \right|
\\
& \ll \left| \sum_{n\leq n_0} {n\choose m} \lambda^{n-m} 
\tau_s^\C(p^n)f_M(z)\right|
\\
& \quad + \sum_{n>n_0} n^{d_0-1} \left|cz+d \right|^{-2\Rea s} \cdot \left|n + 
\frac{1}{c (cz+d)} \right|^{-(M+1+2\Rea s)}.
\end{align*}
Hence, $(H_{s,m}\circ P_M)f(z)$ converges for 
\[
 \Rea s > \frac{d_0-M-1}{2}.
\]
The operator $H_{s,m}\circ P_M$ is nuclear of order $0$ since $P_M$ is bounded.

Since these arguments apply to each $M\in\N_0$, the proof is complete.
\end{proof}

\subsection{Proof of Theorem~\ref{mainthm:strong}}\label{sec:proofcomplete}

Theorem~\ref{mainthm:strong}\eqref{strongi} is just 
Proposition~\ref{nuclear}, and Theorem~\ref{mainthm:strong}\eqref{strongiii} is shown by 
Proposition~\ref{selbergTO}. Theorem~\ref{mainthm:strong}\eqref{strongii} 
follows from Proposition~\ref{mainprop}. The remaining statement 
Theorem~\ref{mainthm:strong}\eqref{strongiv} are implied by 
\eqref{strongi}-\eqref{strongiii} analogously to the case of trivial or unitary 
representations (see, e.\,g., the proof of \cite[Theorem~4.14]{Moeller_Pohl} or of \cite[Theorem~5.1]{Pohl_representation}). \qed

\section{Representations with non-expanding cusp monodromy}\label{sec:repr}

Obviously, there are (non-uniform) Fuchsian groups $\Gamma$ that admit finite-dimensional
representations \textit{without} non-expanding cusp monodromy. In this section we will provide 
examples of finite-dimensional representations \textit{with} non-expanding cusp monodromy. 

For Fuchsian groups $\Gamma$ without any parabolic elements (in particular, for uniform Fuchsian groups), the condition on 
parabolic elements is void and hence any representation of $\Gamma$ 
has non-expanding cusp monodromy. For non-uniform Fuchsian groups $\Gamma$, 
obvious examples of representations with non-expanding cusp monodromy are 
all unitary representations, or more generally, all representations that are 
unitary in the cusps. We provide more examples, and in particular also non-obvious ones, in what follows.

With Proposition~\ref{restric_represent} below we will show that every representation of 
$\Gamma$ that is the restriction of a finite-dimensional representation of 
$\SL_2(\R)$ has non-expanding cusp monodromy. This result is complemented by Remark~\ref{rem:exists} below that proves that this construction indeed yields infinitely many representations of $\Gamma$.

In Proposition~\ref{prop:NECMnontriv} below we will show that for every geometrically finite Fuchsian group 
$\Gamma$ with at least one parabolic element, the class of finite-dimensional representations with non-expanding cusp monodromy 
contains more elements than provided by Proposition~\ref{restric_represent} and 
the obvious examples from above. More precisely, we will construct a finite-dimensional representation of $\Gamma$ with non-expanding cusp monodromy that is neither unitary in the cusps nor a representation from Proposition \ref{restric_represent} nor a direct sum of representations of smaller dimensions. 

Finally, Proposition~\ref{induced_repr}  below will show for all subgroups $\Lambda$ of 
$\Gamma$ of finite index that all finite-dimensional representations with 
non-expanding cusp monodromy of $\Lambda$ induce to finite-dimensional representations with 
non-expanding cusp monodromy of $\Gamma$.

\begin{prop}\label{restric_represent}
Let $\chi\colon \SL_2(\R) \to \GL(V)$ be a finite-dimensional representation of 
$\SL_2(\R)$ and let $\Gamma$ be a Fuchsian group. If $\chi$ descends to a 
representation of $\PSL_2(\R)$ then 
the restriction $\chi_\Gamma$ of $\chi$ to $\Gamma$ has non-expanding cusp 
monodromy.
\end{prop}

\begin{proof}
Since every finite-dimensional representation of $\SL_2(\R)$ decomposes into a direct sum of irreducible representations (see \cite[Corollary~2.2]{Knapp}),
it suffices to prove the statement for irreducible representations. As 
it is well known (see, e.\,g., \cite[(2.1), Proposition 2.1 (a) and (c), Remark 
after Proposition 2.1 and the first line in the proof of Proposition 
2.1]{Knapp}), up to equivalence, the irreducible finite-dimensional 
representations of $\SL_2(\R)$ are parametrized by non-negative integers. The 
representation $(V_n,\chi_n)$ with parameter $n\in \N_0$ can be realized on the 
complex vector space $V_n$ of the polynomials in two complex variables $x$ and 
$y$ that are homogeneous of degree $n$. The action is then given by 
\begin{equation}\label{chi_n_repres}
\big(\chi_n (g) f \big)\left( \begin{smallmatrix} x  \\  y \end{smallmatrix}  
\right) = f\big(  g^{-1} \left( \begin{smallmatrix} x  \\  y \end{smallmatrix}  
\right) \big), 
\end{equation}
where $g\in \SL_2(\R)$ and $f\in V_n$. 

In order to show that $\chi_n$ has non-expanding cusp monodromy, it is 
sufficient to determine the eigenvalues of $\chi(p)$ for the parabolic element 
$p=\textmat{1}{1}{0}{1}$. A basis for $V_n$ is given by the $n+1$ monomials 
$e_0,\ldots, e_n$ of degree $n$ in $2$ variables:
\[
 e_j\left( \begin{smallmatrix} x  \\  y \end{smallmatrix}  \right)\coloneqq x^j 
y^{n-j}, \quad  j = 0, \ldots, n.
\]
A straightforward calculation shows
\begin{equation}\label{chin_in_basis}
 \chi_n(p)e_j = e_j + \sum_{k=1}^j {j\choose k} e_{j-k}
\end{equation}
for all $j\in\{0,1,\ldots, n\}$, which immediately implies (e.\,g., by 
representing $\chi_n(p)$ with respect to the basis $\{e_0, \ldots, e_n\}$) that 
$1$ is the only eigenvalue of $\chi_n(p)$. Hence $\chi_n$ has non-expanding cusp 
monodromy.
\end{proof}

\begin{remark}\label{rem:exists}
Proposition~\ref{restric_represent} indeed yields finite-dimensional 
representations with non-expanding cusp monodromy of Fuchsian groups. As soon as the 
parameter $n$ of the representation $(V_n,\chi_n)$ (notation as in the proof of 
Proposition~\ref{restric_represent}) is even, \eqref{chi_n_repres} implies that 
\[
 \chi_n\left( \mat{-1}{0}{0}{-1}\right)
\]
acts trivially on all elements of $V_n$. Hence $\chi_n$ descends to a 
representation of $\PSL_2(\R)$.

We remark that, by the same reasoning, the representations $(V_n,\chi_n)$ with 
odd parameter $n$ do not descend to $\PSL_2(\R)$.
\end{remark}

For the next class of examples of representations with non-expanding cusp monodromies we recall the signature of Fuchsian groups and its associated presentations. For details we refer to  \cite[\S 10.4]{Beardon}.

Let $\Gamma$ be a non-elementary geometrically finite Fuchsian group, and let $g$ be the genus of $\Gamma \backslash \h$. Further let $r$ be the number of conjugacy classes of maximal elliptic cyclic subgroups of $\Gamma$ (or, equivalently, the number of conical points of $\Gamma\backslash\h$), and suppose that the orders of these subgroups are $m_1, \ldots, m_r$, respectively. Let $s$ be the number of conjugacy classes of maximal parabolic cyclic subgroups in $\Gamma$ (or, equivalently, the number of cusps of $\Gamma\backslash\h$), and let $t$ be the number of conjugacy classes of maximal boundary hyperbolic cyclic subgroups (or, equivalently, the number of funnels of $\Gamma\backslash\h$). The \emph{signature} of $\Gamma$ is then the tuple 
\begin{equation}\label{eq:signature}
(g;m_1,\ldots ,m_r;s;t).
\end{equation}
Each parameter is a non-negative integer, and $m_j \geq 2$ for all $j\in\{1,\ldots, r\}$. If $\Gamma$ does not have elliptic elements, then we write $(g; 0; s; t)$.

For any $g,r,s,t\in\N$ and any $m_j\in\N_{\geq 2}$ for $j\in\{1,\ldots, r\}$, there exists a non-elementary geometrically finite Fuchsian group $\Gamma$ with signature $(g;m_1,\ldots ,m_r;s;t)$ if and only if 
\[
 2g -2 + s + t + \sum_{j=1}^r \left( 1 - \frac1{m_j}\right)> 0.
\]
See \cite[Theorem~10.4.2]{Beardon}.

If $\Gamma$ is a geometrically finite Fuchsian group with signature $(g;m_1,\ldots ,m_r;s;t)$, then $\Gamma$ has a presentation of the form
\begin{align*}
  \Gamma = \Big\langle a_1,b_1,\ldots, a_g,b_g, e_1,\ldots, e_r, &\, p_1,\ldots, p_s, h_1,\ldots,h_t \ \Big\vert\  e_1^{m_1} = \ldots = e_r^{m_r} 
  \\
  & = \big(\prod_{j=1}^g [a_j,b_j]\big) e_1\cdots e_r p_1\cdots p_s h_1\cdots h_t = 1  \Big\rangle,
\end{align*}
where the elements $a_1,b_1,\ldots, a_g,b_g$ are hyperbolic (being responsible for the genus of $\Gamma$), $e_1,\ldots, e_r$ are elliptic (generators of the maximal elliptic cyclic subgroups), $p_1,\ldots, p_s$ are parabolic (generators of the maximal parabolic subgroups), $h_1,\ldots, h_t$ are hyperbolic elements (generators of the maximal boundary hyperbolic cyclic subgroups).

\begin{prop}\label{prop:NECMnontriv}
Let $\Gamma$ be a geometrically finite Fuchsian group with at least one parabolic element. Then there exists a finite-dimensional representation of~$\Gamma$ with non-expanding cusp monodromy that is neither unitary at cusps nor the restriction of a representation of $\SL_2(\R)$ nor decomposes into a direct sum of representations of smaller dimensions.
\end{prop}

Throughout the proof we will use `direct sum' to mean a `direct sum of representations of smaller dimensions.' 

\begin{proof}
Let $\Gamma$ be a geometrically finite non-uniform Fuchsian group with at least one parabolic element, and suppose that the signature of $\Gamma$ is $(g; m_1,\ldots, m_r; s; t)$. The presence of a parabolic element in $\Gamma$ shows that $s\geq 1$. Further, the signature of $\Gamma$ has to satisfy at least one of the following properties:
\vspace*{-2mm}
\begin{center}
\begin{minipage}{7cm}
\begin{enumerate}[(a)]
\item\label{Gz} $g\geq 1$,
\item\label{Ga} $r\geq 1$ and $m_j \neq 3$ for some $j\in\{1,\ldots, r\}$,
\item\label{Ga3} $r\geq 1$ and $m_j = 3$ for some $j\in\{1,\ldots, r\}$,
 \item\label{Gb} $s\geq 2$,
 \item\label{Gc} $t\geq 1$, or
\item\label{Gcylinder} the signature of $\Gamma$ is $(0;0;1;0)$.
\end{enumerate}
\end{minipage}
\end{center}

To prove the proposition we will construct an (explicit) example of a representation
\[
 \varrho\colon\Gamma\to \GL(V)
\]
on some finite-dimensional vector space $V$ such that $\varrho$ has non-expanding cusp monodromy but is neither unitary at cusps nor the restriction of a representation of $\SL_2(\R)$ nor decomposes into a direct sum. The construction of $\varrho$ will depend on whether $\Gamma$ satisfies \eqref{Gz}, \eqref{Ga}, \eqref{Ga3},  \eqref{Gb}, \eqref{Gc} or \eqref{Gcylinder}. (In case that $\Gamma$ satisfies more than one of these properties, our construction yields more than one example.)

We first discuss the case that $\Gamma$ satisfies~\eqref{Gcylinder}. Then $\Gamma$ is freely generated by a single parabolic element, say $\Gamma = \langle p \rangle$. The representation 
\[
 \varrho\colon\Gamma \to \GL_2(\R),
\]
determined by 
\[
 \varrho(p) \coloneqq \mat{1}{1}{0}{1},
\]
has non-expanding cusp monodromy but is not unitary (at the cusp). Since irreducible two-dimensional representations of $\SL_2(\R)$ do not restrict to $\PSL_2(\R)$ and hence \emph{a fortiori} not to $\Gamma$ (see Remark~\ref{rem:exists}), the representation $\varrho$ is not a restriction of an irreducible representation of $\SL_2(\R)$. Further, $\varrho$ is not a direct sum of two one-dimensional representations because $\varrho(p)$ does not diagonalize. \emph{A fortiori,} $\varrho$ is not the restriction of a representation of $\SL_2(\R)$. Hence, $\varrho$ is a representation of $\Gamma$ of the form  claimed. 
 
We now suppose that $\Gamma$ satisfies any of the properties~\eqref{Gz}-\eqref{Gc}. Then let 
\begin{align*}
  \Gamma = \Big\langle a_1,b_1,\ldots, a_g,b_g, e_1,\ldots, e_r, &\, p_1,\ldots, p_s, h_1,\ldots,h_t \ \Big\vert\  e_1^{m_1} = \ldots = e_r^{m_r} 
  \\
  & = \big(\prod_{j=1}^g [a_j,b_j]\big) e_1\cdots e_r p_1\cdots p_s h_1\cdots h_t = 1  \Big\rangle
\end{align*}
be a presentation of $\Gamma$ that realizes its signature. Thus, the elements $a_1,b_1,\ldots, a_g,b_g$ are hyperbolic, $e_1,\ldots, e_r$ are elliptic, $p_1,\ldots, p_s$ are parabolic, $h_1,\ldots, h_t$ are boundary hyperbolic elements. Throughout we assume without loss of generality that 
\[
 p_1 = \bmat{1}{1}{0}{1}.
\]
Obviously, it suffices to define $\varrho$ on the set of generators 
\[
\mc G\coloneqq \{a_1,b_1,\ldots, a_g,b_g, e_1,\ldots, e_r, p_1,\ldots, p_s, h_1,\ldots, h_t\}
\]
for $\Gamma$, checking that it obeys all relations in the presentation of $\Gamma$ from above, and then extend it (in the unique possible way) to all of $\Gamma$ enforcing homomorphy.

We suppose first that $\Gamma$ satisfies \eqref{Gz}. In this case, we set $V\coloneqq \C^2$.
We fix $\lambda\coloneqq e^{i \pi / 2}$, $\mu\coloneqq e^{i \pi / 6}$. We  set
\begin{align}
\varrho(a_1) &\coloneqq \mat{\lambda}{1}{0}{\lambda^{-1}},\qquad \varrho(b_1) \coloneqq \mat{\mu}{1}{0}{\mu^{-1}}, \nonumber
\\[2mm]
\varrho(p_1)  &\coloneqq \mat{1}{\lambda\mu\big( \lambda^{-1}-\lambda + \mu-\mu^{-1}\big) }{0}{1} \label{def_chip}
\intertext{and}
\varrho(g) & \coloneqq \id_{\C^2} \quad\text{for all $g\in \mc G\smallsetminus\{a_1,b_1,p_1\}$.} \nonumber
\end{align}
Then the relations in the presentation of $\Gamma$ are satisfied. Extending $\varrho$ to all of $\Gamma$ by enforcing homomorphy then yields a well-defined finite-dimensional representation of $\Gamma$. From~\eqref{def_chip} we read off that $\varrho$ has non-expanding cusp monodromy. Further, since 
\[
\lambda\mu\big( \lambda^{-1}-\lambda + \mu-\mu^{-1}\big) =\frac12\big(i + \sqrt{3}\big)\not=0,
\]
the representation~$\varrho$ is not unitary at cusps. Since $\varrho$ is two-dimensional, it cannot be the restriction of an irreducible representation of $\SL_2(\R)$ (see Remark~\ref{rem:exists}). Further, since $\varrho(p_1)$ is not diagonalizable, $\varrho$ is not the direct sum of two one-dimensional representations. Therefore $\varrho$ is of the type of claimed.

We now suppose that $\Gamma$ satisfies \eqref{Ga}, \eqref{Gb} or \eqref{Gc}. In each case we let $(V_2,\chi_2)$ be the representation of $\SL_2(\R)$ from \eqref{chi_n_repres} with $n=2$ and let $\chi'_2$ be the restriction of $\chi_2$ to $\Gamma$. We set $V\coloneqq V_2$ and will define $\varrho$ by modifying $\chi'_2$ as explained in what follows.

\begin{enumerate}[{\rm (i)}]
\item \label{case i} If $\Gamma$ satisfies \eqref{Ga} then let $\tau$ be an $m_1$-th root of unity such that $\tau^3\not=1$. We set 
\begin{align*}
  \varrho(e_1) &\coloneqq \tau^{-1}\cdot \chi'_2(e_1),\qquad \varrho(p_1)\coloneqq \tau\cdot \chi'_2(p_1),
  \intertext{and}
  \varrho(g) &\coloneqq \chi'_2(g)\qquad\text{for $g\in\mc G\smallsetminus\{e_1,p_1\}$.}
\end{align*}
\item \label{case ii} If $\Gamma$ satisfies \eqref{Gb} then we pick $\tau\in\C$ such that $|\tau|=1$ and $\tau^3\not=1$. We set
\begin{align*}
  \varrho(p_1) &\coloneqq \tau\cdot \chi'_2(p_1),\qquad \varrho(p_2)\coloneqq \tau^{-1}\cdot \chi'_2(p_2),
  \intertext{and}
  \varrho(g) &\coloneqq \chi'_2(g)\qquad\text{for $g\in\mc G\smallsetminus\{p_1,p_2\}$.}
\end{align*}
\item \label{case iii} If $\Gamma$ satisfies \eqref{Gc} then we again pick $\tau\in\C$ such that $|\tau|=1$ and $\tau^3\not=1$. We set
\begin{align*}
  \varrho(p_1) &\coloneqq \tau\cdot \chi'_2(p_1),\qquad \varrho(h_1)\coloneqq \tau^{-1}\cdot \chi'_2(h_1),
  \intertext{and}
  \varrho(g) &\coloneqq \chi'_2(g)\qquad\text{for $g\in\mc G\smallsetminus\{p_1,h_1\}$.}
\end{align*}
\end{enumerate}
In each case, one easily checks that the relations in the presentation of $\Gamma$ are satisfied. Further, \eqref{chin_in_basis} and $\varrho(p_1) = \tau \chi_2'(p_1)$ imply that $\varrho(p_1)$ is not unitary. Hence $\varrho$ is not unitary at cusps. Since $\varrho(p_1)$ consists of a single Jordan block (see~\eqref{chin_in_basis}), $\varrho$ does not decompose into a direct sum. 

Finally, if $\varrho$ was the restriction of an irreducible representation of $\SL_2(\R)$, then $\varrho$ would be equivalent to $\chi'_2$. Then there would exist an intertwiner $\phi\colon V_2 \to V_2$ such that $\phi\circ\varrho = \chi_2'\circ\phi$. This would imply 
\[
1=\det(\chi'_2(p_1)) = \det(\varrho(p_1))=\tau^3 \det(\chi'_2(p_1))= \tau^3,
\]
which contradicts $\tau^3 \neq  1$.

If $\Gamma$ satisfies \eqref{Ga3} then we perform a similar construction but start with the representation $(V_4,\chi_4)$ of $\SL_2(\R)$ (from~\eqref{chi_n_repres}). Let $\chi'_4$ denote its restriction to $\Gamma$, and let $\tau$ be a third root of unity such that $\tau^5\not=1$. We set 
\begin{align*}
\varrho(e_1) &\coloneqq \tau^{-1}\cdot \chi'_4(e_1),\qquad \varrho(p_1)\coloneqq \tau\cdot \chi'_4(p_1),
\intertext{and}
\varrho(g) &\coloneqq \chi'_4(g)\qquad\text{for $g\in\mc G\smallsetminus\{e_1,p_1\}$.}
\end{align*}
It follows that $\varrho$ is not unitary at cusps. Since $\varrho(p_1)$ again consists of a single Jordan block, it does not decompose into a direct sum. Further, if $\varrho$ was a restriction of an irreducible representation of $\SL_2(\R)$, then it would be equivalent to $\chi'_4$ and hence, analogously, to above
\[
1=\det(\chi'_4(p_1)) = \det(\varrho(p_1)) = \tau^{5} \det(\chi'_4(p_1)) = \tau^5,
\]
which is a contradiction. This completes the proof.
\end{proof}

We end this section with the proof that for representations the property of non-expanding cusp monodromy is stable under induction.

\begin{prop}\label{induced_repr}
Let $\Gamma$ be a Fuchsian group, $\Lambda$ a subgroup of $\Gamma$ of finite 
index and $\eta\colon\Lambda\to \GL(V)$ a finite-dimensional representation with 
non-expanding cusp monodromy. Then the induced representation 
$\chi\coloneqq\Ind_\Lambda^\Gamma\eta$ of $\Gamma$ has non-expanding cusp 
monodromy. 
\end{prop}

\begin{proof}
Let $\wt\eta\colon \Gamma\to \End(V)$ be defined by
\[
 \wt\eta(g) \coloneqq
 \begin{cases}
  \eta(g) & \text{if $g\in\Lambda$} 
  \\
  0 & \text{otherwise.}
 \end{cases}
\]
Let $n\coloneqq [\Gamma:\Lambda]$, and let $R=\{h_1,\ldots, h_n\}$ be a complete set 
of representatives in $\Gamma$ for the right cosets $\Lambda\backslash\Gamma$ 
(thus, $\Gamma=\bigcup_{j=1}^n \Lambda h_j$ as a disjoint union). Then $\chi$ is 
(equivalent to) the representation on $V^n\cong \bigoplus_{j=1}^n V$ given by
\begin{equation}\label{model_indrep}
 \chi(g) = \Big( \wt\eta\big(h_igh_j^{-1}\big) \Big)_{i,j=1,\ldots, n} \qquad 
(g\in\Gamma).
\end{equation}
Let $p\in\Gamma$ be parabolic. Since $R$ is a complete set of representatives 
for $\Lambda\backslash\Gamma$, for each $i_0\in \{1,\ldots, n\}$ there exists 
exactly one $j_0\in\{1,\ldots, n\}$, and for each $j_0\in\{1,\ldots, n\}$ there 
exists exactly one $i_0\in\{1,\ldots,n\}$ such that 
$h_{i_0}ph_{j_0}^{-1}\in\Lambda$. In other words, in the presentation 
\eqref{model_indrep}, the endomorphism $\chi(p)$ is given by a `permutation  
matrix' where instead of the entry $1$ (at position $(i,j)$) we have 
$\eta\big(h_iph_j^{-1}\big)$. It follows that there exists $k\in\N$ (indeed $k\leq n!$) such that 
\[
 \chi\big(p^{k}\big) = \chi(p)^k = \diag\Big(\eta(h_1p^{k}h_1^{-1}\big),\ldots, 
\eta\big(h_np^{k}h_n^{-1}\big) \Big).
\]
Since all eigenvalues of the endomorphisms $\eta(h_jp^{k}h_j^{-1})$, 
$j=1,\ldots, n$, are of absolute value $1$, so are all eigenvalues of 
$\chi(p)^{k}$, and hence all eigenvalues of $\chi(p)$. Thus, $\chi$ has 
non-expanding cusp monodromy.
\end{proof}

\section{No convergence beyond non-expanding cusp monodromy}

In the preceding sections we showed that if $\chi$ is a representation with non-expanding cusp monodromy, then the infinite product~\eqref{pre_def} used for the definition of the $\chi$-twisted Selberg zeta function converges absolutely in a half space, and has a meromorphic continuation to all of $\C$. In this section we will discuss the natural question if these results hold as well for some representations $\chi$ not having non-expanding cusp monodromy. 

The infinite product~\eqref{pre_def} contains the product over the primitive periodic geodesics or, equivalently, the $\Gamma$-conjugacy classes of the primitive hyperbolic elements in~$\Gamma$, which does not have a natural order. Therefore, absolute convergence is arguably the most natural notion of convergence when considering \eqref{pre_def}.

We will show that as soon as $\chi$ does not have non-expanding cusp monodromy and the infinite product~\eqref{pre_def} is not void then the latter does not converge absolutely, and hence the notion of twisted Selberg zeta function has no meaning (within the realms of absolute convergence). This observation shows that the property of non-expanding cusp monodromy is indeed the weakest requirement on twisting representations we can allow for a theory of Selberg zeta functions.

\begin{prop}\label{prop:thm_sharp}
Let $\Gamma$ be a geometrically finite Fuchsian group with at least one hyperbolic element, and let $\chi\colon\Gamma\to\GL(V)$ be a finite-dimensional representation of $\Gamma$ that does not have non-expanding cusp monodromy. Then the infinite product~\eqref{pre_def} does not converge absolutely for any $s\in\C$.
\end{prop}

\begin{proof}
Since $\chi$ does not have non-expanding cusp monodromy, there exists a parabolic element $p\in\Gamma$ such that not all eigenvalues of $\chi(p)$ are of absolute value $1$. Without loss of generality (by possibly conjugating $\Gamma$, and taking the inverse of $p$) we may assume that 
\[
 p=\bmat{1}{1}{0}{1}
\]
and that $\chi(p)$ has an eigenvalue $\lambda$ with $|\lambda|>1$. Further, for readibility, we assume that $\chi(p)$ has a single Jordan block. The general case of several Jordan blocks is easily handled by a straightforward extension of the argumentation that follows.

By hypothesis, $\Gamma$ has at least one hyperbolic element, hence the associated hyperbolic surface $X=\Gamma\backslash\h$ has at least one (primitive) periodic geodesic, and the infinite product \eqref{pre_def} is non-empty. Let 
\[
 h=\bmat{a}{b}{c}{d}
\]
be a hyperbolic element in $\Gamma$. In what follows we will show that $(hp^n)_{n\in\N}$ contains a subsequence of hyperbolic elements such that even if the infinite product \eqref{pre_def} is restricted to these hyperbolic elements (or rather its underlying primitive elements), it does not converge. Geometrically, it means that we restrict the considerations to a periodic geodesic winding arbitrarily often around the cusp associated to $p$. The weights that $\chi$ assigns to these geodesics are so large that convergence of \eqref{pre_def} is prevented.

By the hypothesis above on~$p$, the point~$\infty$ is cuspidal and hence the element~$c$ in~$h$ does not vanish. A straightforward calculation shows that for all $n\in\N$ we have
\[
 \tr\big(hp^n\big) = \tr h + cn.
\]
This shows that there exists $n_0\in\N$ such that the elements of~$\{hp^n \mid n\geq n_0\}$ are all hyperbolic and pairwise non-conjugate in~$\Gamma$.

As discussed in Section~\ref{sec:conv}, if the infinite product \eqref{pre_def} converges absolutely for $s\in\C$ then 
\[
 \sum_{[g]\in [\Gamma]_\hyp} \frac{1}{m(g)} \frac{N(g)^{-s}}{1-N(g)^{-1}} \tr\chi(g)
\]
converges absolutely. In particular, for $n_0\in\N$ sufficiently large, the sum
\begin{equation}\label{rest_sum}
 \Sigma(s) \coloneqq \sum_{n\geq n_0} \frac{1}{m(hp^n)} \frac{N\big(hp^n\big)^{-\Rea s}}{1-N\big(hp^n\big)^{-1}} |\tr\chi\big(hp^n\big)|
\end{equation}
converges for these $s\in\C$. In what follows we show that the sum \eqref{rest_sum} does not converge. 

To that end, we start by estimating $N\big(hp^n\big)$ and $m\big(hp^n\big)$ for $n\in\N$. The eigenvalues of $hp^n$ are 
\[
 \mu_\pm = \frac{a+d+cn \pm \sqrt{(a+d+cn)^2 - 4}}{2}.
\]
Thus (see \eqref{norm_h})
\begin{equation}\label{est_norm}
 N\big(hp^n\big) \asymp n^2
\end{equation}
with implied constants independent of $n$. Let 
\[
 \eps_0 \coloneqq \min\{ N(h_0) \mid \text{$h_0\in\Gamma$ hyperbolic} \}.
\]
We note that this minimum indeed exists. Suppose that $h_n$ is the primitive hyperbolic element underlying $ph^n$, and let $m\coloneqq m\big(hp^n\big)$. Hence $h_n^{m} = hp^n$ and therefore
\[
\eps_0^m \leq N(h_n)^m = N\big(h_n^m\big) = N\big(hp^n\big) \asymp n^2.
\]
It follows that 
\begin{equation}\label{est_level}
 m\big(hp^n\big) = m \ll \frac1{\log n}
\end{equation}
with implied constant independent of $n$. Now let $d\coloneqq\dim V$ and note that there exist $a_0,\ldots, a_{d-1}\in\C$ such that for all $n\in\N$ sufficiently large we have
\begin{equation}\label{est_tr}
\tr\chi\big(hp^n\big) = \lambda^n\left( a_0 + a_1 n +\ldots + a_{d-1}n^{d-1} \right)  \gg n^{d-1} |\lambda|^n 
\end{equation} 
with implied constant independent of $n$. Using \eqref{est_norm}-\eqref{est_tr} in \eqref{rest_sum} we find that 
\begin{align*}
 \Sigma(s) &\gg \sum_{n\geq n_0} \frac{1}{\log n} \frac{n^{-2\Rea s}}{1-n^{-2}} n^{d-1} |\lambda|^n
 \\
 & \gg \sum_{n\geq n_0} n^k |\lambda|^n
\end{align*}
for some $k\in\R$. Since $|\lambda|>1$, the last sum diverges. This completes the proof.

\end{proof}

\section{Venkov--Zograf factorization formulas}\label{sec:factor}

Venkov and Zograf \cite{Venkov_Zograf} showed that Selberg zeta functions for 
geometrically finite Fuchsian groups twisted with unitary finite-dimensional 
representations satisfy certain factorization formulas. The purpose of this 
section is to provide, with Theorem~\ref{thm:factor} below, these factorization 
formulas for twists with non-expanding cusp monodromy.

To be precise, the proof of the factorization formulas in \cite{Venkov_Zograf} takes advantage of the Selberg trace formula, which restricts the proof method from \cite{Venkov_Zograf} to cofinite Fuchsian groups and unitary representations. However, the 
simplification (which essentially consists in omitting some unnecessary steps) 
sketched in \cite[Theorem~7.2]{Venkov_book2} applies without changes to 
non-cofinite Fuchsian groups as well, even though 
\cite[Theorem~7.2]{Venkov_book2} is stated for cofinite Fuchsian groups only. 

In combination with Proposition~\ref{induced_repr} the latter proof applies also 
to all finite-dimensional representations with non-expanding cusp monodromy, 
again without any changes. For the convenience of the reader we provide a 
complete proof.

\begin{thm}\label{thm:factor}
Let $\Gamma$ be a geometrically finite Fuchsian group, and $\Lambda$ a subgroup 
of $\Gamma$ of finite index. Let $\chi\colon \Gamma\to \GL(V)$ be a 
finite-dimensional representation with non-expanding cusp monodromy. 
\begin{enumerate}[{\rm (i)}]
\item\label{factori} If $\chi=\bigoplus\limits_{j=1}^m\chi_j$ decomposes into a 
direct (finite) sum of representations $\chi_j$ of $\Gamma$ with non-expanding 
cusp monodromy, then 
\[
 Z(s,\Gamma,\chi) = \prod_{j=1}^m Z(s,\Gamma,\chi_j).
\]
\item\label{factorii} If $\eta\colon\Lambda\to \GL(V)$ is a representation with 
non-expanding cusp monodromy, then 
\[
 Z(s,\Lambda,\eta) = Z(s,\Gamma,\Ind_\Lambda^\Gamma\eta).
\]
Recall from Proposition~\ref{induced_repr} that $\Ind_\Lambda^\Gamma\eta$ has 
indeed non-expanding cusp monodromy.
\item\label{factoriii} If $\Lambda$ is normal in $\Gamma$, then 
\[
Z(s,\Lambda,{\bf 1}) = \prod_{\eta\in \wh{(\Lambda\backslash\Gamma)}} 
Z(s,\Gamma,\eta)^{\dim\eta}, 
\]
where $\wh{(\Lambda\backslash\Gamma)}$ denotes the unitary dual of 
$\Lambda\backslash\Gamma$.
\end{enumerate}
\end{thm}

\begin{proof}
The proof of \eqref{factori} follows immediately from \eqref{pre_def} in the 
domain of convergence of the infinite products and extends to all of $\C$ by 
uniqueness of meromorphic continuations. In the case that $\Gamma$ admits a 
strict transfer operator approach, an alternative transfer operator based proof 
of \eqref{factori} is possible and straightforward since the transfer operators 
decompose according to the representation. We refer to 
\cite[Section~6]{Pohl_representation} for details.

Further, \eqref{factoriii} follows immediately from \eqref{factori} and 
\eqref{factorii} by considering the decomposition of the induced representation 
$\Ind_\Lambda^\Gamma {\bf 1}$ into irreducibles.

It remains to prove \eqref{factorii}. To that end we show that the logarithms of 
the two Selberg zeta functions coincide on the domain  on which both of them are 
given by the infinite products in \eqref{pre_def}. Uniqueness of 
meromorphic continuations then completes the proof. We start with a few 
preparations.

We define, as in the proof of Proposition~\ref{induced_repr}, the map $\wt\eta\colon \Gamma\to \End(V)$ by
\[
 \wt\eta(g) \coloneqq 
 \begin{cases}
  \eta(g) & \text{if $g\in\Lambda$}
  \\
  0 & \text{otherwise.}
 \end{cases}
\]
For $g,h\in\Gamma$ we write $g\sim_\Gamma h$ or $g\sim_\Lambda h$ if 
$g$ and $h$ are $\Gamma$-conjugate or $\Lambda$-conjugate, respectively.

Let $\chi\coloneqq \Ind_\Lambda^\Gamma\eta$, and let $R$ be a complete set of 
representatives for the right cosets $\Lambda\backslash\Gamma$. Thus,
\[
 \Gamma = \bigcup_{h\in R} \Lambda h \qquad \text{(disjoint union)}.
\]
Note that $R$ is finite. For any $g\in\Gamma$ we have 
\begin{equation}\label{Frob}
 \tr\chi(g) = \sum_{h\in R} \tr\wt\eta\big(hgh^{-1}\big)
\end{equation}
by Frobenius formula (also known as Mackey formula). 

Let $p_0\in\Gamma$ be $\Gamma$-primitive hyperbolic, and consider
\[
 \Lambda(p_0) \coloneqq \{ gp_0^kg^{-1} \mid g\in\Gamma,\ k\in\N\} \cap \Lambda.
\]
Then \cite[Lemma~2.3]{Venkov_Zograf} shows that whenever $gp_0^kg^{-1}\in 
\Lambda(p_0)$ for some $g\in\Gamma$, $k\in\N$ then there exists a 
$\Lambda$-primitive element $q_0\in\Lambda$ and a unique $e_0\in\N$ such that 
$q_0\sim_\Gamma p_0^{e_0}$ and $e_0$ divides $k$. Thus, $q_0\in \Lambda(p_0)$. 

By slight abuse of notation, let $[\Lambda(p_0)]_\prim$ denote the set of 
$\Lambda$-conjugacy classes of $\Lambda$-primitive hyperbolic elements in 
$\Lambda(p_0)$. By \cite[Lemmas~2.1-2.2]{Venkov_Zograf},  $[\Lambda(p_0)]_\prim$ 
is finite. Let 
\[
 R(p_0) \coloneqq \{ q_1,\ldots, q_M\} 
\]
be a complete set of representatives of $[\Lambda(p_0)]_\prim$ in $\Lambda$ 
(thus, $q_1,\ldots,q_M$ are $\Lambda$-primitive hyperbolic and pairwise  
non-conjugate in $\Lambda$, and each element of $[\Lambda(p_0)]_\prim$ has a representative in $R(p_0)$). For $j\in \{1,\ldots,M\}$ let $e_j\in\N$ such that 
$q_j\sim_\Gamma p_0^{e_j}$, and let 
\[
 R_j \coloneqq \{ h\in R \mid q_j \sim_\Lambda hp_0^{e_j}h^{-1} \}.
\]
By \cite[Lemma~2.1]{Venkov_Zograf}, $|R_j| = e_j$. It follows that for all 
$k\in\N$ with $e_j\mid k$ we have
\[
 e_j \tr \eta\big(q_j^{k/e_j}\big)= \sum_{h\in R_j} \tr 
\eta\big(hp_0^kh^{-1}\big).
\]
Thus, for all $k\in\N$ it follows
\[
 \sum_{\stackrel{j=1}{e_j\mid k}}^M e_j \tr \eta\big(q_j^{k/{e_j}}\big) = 
\sum_{h\in R} \tr \wt\eta\big(hp_0^kh^{-1}\big)  = \tr\chi\big(p_0^k\big).
\]
Hence
\begin{align*}
\sum_{k=1}^\infty \frac1k \frac{N\big(p_0^k\big)^{-s}}{1-N\big(p_0^k\big)} \tr 
\chi\big(p_0^k\big)
& = \sum_{k=1}^\infty \sum_{\stackrel{j=1}{e_j\mid k}}^M \frac{e_j}{k} 
\frac{N\big(p_0^k\big)^{-s}}{1-N\big(p_0^k\big)} \tr 
\eta\big(q_j^{\frac{k}{e_j}}\big)
\\
& = \sum_{j=1}^M \sum_{m=1}^\infty \frac{e_j}{e_jm} 
\frac{N\big(p_0^{e_jm}\big)^{-s}}{1-N\big(p_0^{e_jm}\big)} 
\tr\eta\big(q_j^m\big)
\\
& = \sum_{j=1}^M \sum_{m=1}^\infty \frac1m  
\frac{N\big(q_j^m\big)^{-s}}{1-N\big(q_j^m\big)^{-1}} \tr\eta(q_j^m).
\end{align*}
Thus, for $\Rea s$ sufficiently large, 
\begin{align*}
\log Z(s,\Gamma,\chi) & = - \sum_{[p_0]\in[\Gamma]_\prim} \sum_{k=1}^\infty 
\frac{1}{k}\frac{N\big(p_0^k\big)^{-s}}{1-N\big(p_0^k\big)^{-1}} 
\tr\chi\big(p_0^k\big)
\\
& = - \sum_{[p_0]\in[\Gamma]_\prim} \sum_{[q]\in[\Lambda(p_0)]_\prim} 
\sum_{m=1}^\infty \frac1m \frac{N\big(q^m\big)^{-s}}{1-N\big(q^m\big)^{-1}} 
\tr\eta\big(q^m\big)
\\
& = -\sum_{[g]\in[\Lambda]_p} \sum_{m=1}^\infty \frac1m 
\frac{N\big(g^m\big)^{-s}}{1-N\big(g^m\big)^{-1}} \tr\eta\big(g^m\big)
\\
& = \log Z(s,\Lambda,\eta).
\end{align*}
This completes the proof. 
\end{proof}

\appendix

\section{Meromorphic continuation of the derivatives of the Lerch transcendent}\label{sec:lerch}

In Section~\ref{sec:merom} we used the Lerch transcendent 
\begin{equation}\label{lerch1}
 \Phi(s,\lambda,w) =\sum_{n=0}^\infty \frac{\lambda^n}{(n+w)^s}
\end{equation}
and its partial derivatives 
\begin{equation}\label{lerch2}
 \Phi_m(s,\lambda,w) =\frac{1}{m!} 
\frac{\partial^m\Phi}{\partial\lambda^m}(s,\lambda,w) = \sum_{n=m}^\infty 
{n\choose m} \frac{\lambda^{n-m}}{(n+w)^s} \qquad (m\in\N_0)
\end{equation}
for the proof of the meromorphic continuability of the transfer operator families. In particular we took advantage of the convergence of $\Phi_m$, $m\in\N_0$, their extensions to meromorphic or even entire functions to all of $\C$ in the variable $s$ for any $w\in\C\smallsetminus(-\N_0)$, and their continuity as a function of the two variables $(s,w)$ for any $\lambda\in\C$, $|\lambda|=1$. Proposition~\ref{prop:lerch} below contains a precise statement of the properties we used. 

For $\Phi = \Phi_0$, proofs of these properties can be found in the literature. For $\Phi_m$, $m>0$, they can be deduced from those of $\Phi$ (at least in the case $\lambda\not=1$) or shown in analogy to those of $\Phi$. For the convenience of the reader, we provide a complete proof of Proposition~\ref{prop:lerch}.

For this proof we take an approach via Mellin transforms and functional equations. A descent reference for this approach is Zagier's paper \cite{Zagier_mellin}.

\begin{prop}\label{prop:lerch}
Let $\lambda\in\C$, $|\lambda|=1$, $m\in\N_0$. Then the infinite sum $\Phi_m(\cdot,\lambda,\cdot)$ in \eqref{lerch2} is convergent on 
\[
 \{s\in\C\mid \Rea s > m+1\} \times \big(\C\smallsetminus(-\N_0)\big),
\]
defines there a continuous function, and for any fixed $w\in\C\smallsetminus(-\N_0)$, the map $\Phi_m(\cdot,\lambda,w)$ is holomorphic on 
\[
\{s\in\C\mid\Rea s > m+1\}.
\] 
Furthermore, $\Phi_m(\cdot,\lambda,\cdot)$ is holomorphic on 
\[
 \{s\in\C\mid \Rea s > m+1\} \times \big(\C\smallsetminus(-\infty,0]\big).
\]
For $\lambda\not=1$, the map $\Phi_m(\cdot,\lambda,\cdot)$ extends to a holomorphic function on 
\[
\C\times\big(\C\smallsetminus(-\infty,0]\big).
\]
For $\lambda=1$, it extends to a meromorphic function on $\C\times\big(\C\smallsetminus(-\infty,0]\big)$. For any fixed $w\in\C\smallsetminus(-\infty,0]$, the poles of $\Phi_m(\cdot, 1,w)$ are contained in $\{1,\ldots, m+1\}$, all poles are simple, and $m+1$ is a pole.

If a continuous extension of the complex logarithm to $\C\smallsetminus (-\N_0)$ is fixed, then the meromorphic extendability in $s$ generalizes to all fixed $w \in \C \smallsetminus (-\N_0)$,  and the meromorphic extension is continuous as a function of the two variables $(s,w)$.
\end{prop}

\begin{proof}
Throughout let $\lambda\in\C$, $|\lambda|=1$, and let $m\in\N_0$. Definitions \eqref{lerch1} and \eqref{lerch2} clearly show that $\Phi_m(\cdot,\lambda,\cdot)$ converges on 
\[
 \mc D_m\coloneqq \{ s\in \C \mid \Rea s > m+1\} \times \big(\C\smallsetminus (-\N_0)\big),
\]
defines a continuous function on this domain, and for any fixed $w\in\C\smallsetminus (-\N_0)$, the map $\Phi_m(\cdot,\lambda,w)$ is holomorphic on $\{ s\in \C \mid \Rea s > m+1\}$. Furthermore, the map $\Phi_m(\cdot,\lambda,\cdot)$ is holomorphic on 
\[
 \{ s\in \C \mid \Rea s > m+1\}\times \big(\C\smallsetminus (-\infty,0]\big).
\]

For the proof of the claimed meromorphic continuability we first note that for all $k\in\N_0$ and all $(s,w)\in \mc D_0$ we have
\[
 \Phi(s,\lambda,w) = \lambda^k \Phi(s,\lambda,w+k) + \sum_{n=0}^{k-1} \frac{\lambda^n}{(n+w)^s}.
\]
Moreover, for every $n\in\N_0$, the map
\[
 \C\times \big(\C\smallsetminus (-\infty,0]\big) \to \C,\quad (s,w)\mapsto \frac{\lambda^n}{(n+w)^s}
\]
is holomorphic. 

Combining this functional equation with the definition \eqref{lerch2} of $\Phi_m$ as a partial derivative of $\Phi$ yields that for all $k\in\N_0$ and all $(s,w)\in\mc D_m$, 
\begin{equation}\label{funceq_lerch}
 \Phi_m(s,\lambda,w) = \sum_{n=0}^m \binom{k}{n}\lambda^{k-n} \Phi_{m-n}(s,\lambda,w+k) + \sum_{n=0}^{k-1} \binom{n}{m} \frac{\lambda^{n-m}}{(n+w)^s}.
\end{equation}
Suppose for a moment that we have already shown (which will be done further below) that for any $v\in\C$ with $\Rea v>0$ and any $p\in\N_0$, the map $\Phi_p(\cdot,\lambda,v)$ extends, for $\lambda\not=1$, to an entire function, and, for $\lambda=1$, to a meromorphic function on all of $\C$ all of whose poles are contained in $\{1,\ldots, p+1\}$, are (at most) simple, and $p+1$ is indeed a pole.  Suppose further that we have already shown that these extensions are meromorphic in the variables $(s,v)$ on 
\[
 \big(\C\smallsetminus\{\text{poles}\}\big)\times \{v\in\C\mid\Rea v > 0 \}.
\]
Applying \eqref{funceq_lerch} with sufficiently large (and varying) $k\in\N_0$ proves the statement of the proposition for $w$ restricted to $\C\smallsetminus (-\infty,0]$. Fixing a continuous extension of the complex logarithm to $\C\smallsetminus (-\N_0)$ then completes the full proof of the proposition.

Thus it remains to establish the proposition for the domain of the variable $w$ restricted to $\Rea w>0$. We start by expressing $\Phi_m$ by a Mellin transform of a certain highly regular function with good decay properties. Throughout let $w\in\C$ with $\Rea w>0$, and $s\in\C$ with $\Rea s>m+1$. The integral representation
\[
 \Gamma(r) = \int_0^\infty u^{r-1}e^{-u}\, du, \qquad \Rea r > 0
\]
of the Gamma function yields for all $n\in\N_0$ that 
\[
 \frac{1}{(n+w)^r} = \frac{1}{\Gamma(r)} \int_0^\infty e^{-t(n+w)}t^{r-1}\, dt
\]
(initially only for $w\in\R_{>0}$ and then by holomorphy and identity theorem also for $w\in\C$, $\Rea w>0$). It follows that (note that $|\lambda e^{-t}|<1$ for $t>0$)
\begin{align*}
\Phi_m(s,\lambda,w) & = \sum_{n=m}^\infty \binom{n}{m} \frac{\lambda^{n-m}}{(n+w)^s} 
\\
& = \frac{1}{\Gamma(s)} \sum_{n=m}^\infty \binom{n}{m} \lambda^{n-m} \int_0^\infty t^{s-1} e^{-t(n+w)}\, dt
\\
& = \frac{1}{\Gamma(s)} \int_0^\infty t^{s-1} e^{-(m+w)t} \sum_{n=m}^\infty \binom{n}{m} \big(\lambda e^{-t}\big)^{n-m}\, dt
\\
& = \frac{1}{\Gamma(s)} \int_0^\infty t^{s-1} \frac{e^{-(w+m)t}}{\big(1-\lambda e^{-t}\big)^{m+1}}\, dt
\\
& = \frac{1}{\Gamma(s)} \int_0^\infty t^{s-1} \frac{e^{-(w-1)t}}{\big(e^t - \lambda\big)^{m+1}}\, dt.
\end{align*}
Let $\varphi_{\lambda,w}\colon\R\to\C$, 
\[
 \varphi_{\lambda,w}(t) \coloneqq \frac{e^{-(w-1)t}}{\big(e^t - \lambda\big)^{m+1}},
\]
and let $\wt\varphi_{\lambda,w}\colon \{s\in\C\mid \Rea s > m+1\}\to\C$,
\[
 \wt\varphi_{\lambda,w}(s) \coloneqq \int_0^\infty t^{s-1} \varphi_{\lambda,w}(t)\, dt
\]
denote the Mellin transform of $\varphi_{\lambda,w}$. Thus
\[
 \Phi_m(s,\lambda, w) = \Gamma(s)^{-1} \wt\varphi_{\lambda,w}(s).
\]
Since $\Rea w > 0$, one easily sees that $\varphi_{\lambda,w}$ is of rapid decay as $t\to \infty$. Moreover, as $t\to 0$, the function $\varphi_{\lambda,w}$ has an asymptotic expansion
\[
 \varphi_{\lambda,w}(t) \sim \sum_{n=-(m+1)}^\infty a_n(\lambda,w) t^n
\]
with suitable coefficients $a_n(\lambda, w)\in\C$. For $\lambda\not=1$, $\varphi_{\lambda,w}$ is smooth at $0$, and therefore
\[
 a_n(\lambda,w) = 0 \quad\text{ if $\lambda\not=1$ and $n\in\{-(m+1),\ldots, -1\}$.}
\]
For $\lambda=1$, $\varphi_{\lambda,w}$ has a pole of order $m+1$ at $0$. 

For any $N\in\Z$, $N>-(m+1)$, we find
\begin{equation}\label{phi_decomp}
\begin{gathered}
\wt\varphi_{\lambda,w}(s) = \int_0^\infty t^{s-1} \left( \varphi_{\lambda,w}(t) - \sum_{n=-(m+1)}^{N-1} a_n(\lambda,w) t^n\right)\, dt + \sum_{n=-(m+1)}^{N-1} \frac{a_n(\lambda,w)}{s+n}
\\
 \quad + \int_1^\infty t^{s-1}\varphi_{\lambda,w}(t)\,dt. 
\end{gathered}\end{equation}
The last integral on the right hand side defines a function in $(s,w)$ that is holomorphic on 
\[
 \C \times \{w\in\C\mid\Rea w > 0\},
\]
the first integral defines a function holomorphic on 
\[
 \{s\in\C\mid \Rea s > m+1-N\} \times \{w\in\C\mid\Rea w > 0\}.
\]
Thus, $(s,w)\mapsto\wt\varphi_{\lambda,w}(s)$ extends to a meromorphic function on 
\[
\{s\in\C\mid \Rea s > m+1-N\} \times \{w\in\C\mid\Rea w > 0\}.
\]
For any fixed $w\in\C$, $\Rea w>0$, the poles of $\wt\varphi_{\lambda,w}$ are contained in
\[
 \{ -(N-1), \ldots, m+1\}.
\]
For $n\in\{-(m+1),\ldots, N-1\}$, the order of $s=-n$ as a pole is at most $1$, its residue is $a_{n}(\lambda,w)$. Thus, if $\lambda\not=1$ then the set of poles shrinks to 
\[
 \{ -(N-1),\ldots, 0\}.
\]
If $\lambda=1$ then $a_{-(m+1)} = 1\not=0$, and $s=m+1$ is indeed a pole. 

This argumentation obviously applies to any $N$, and hence we find that $\wt\varphi_{\lambda,w}$ has a meromorphic continuation to all of $\C$ with all poles contained in $\Z_{\leq  m+1}$, and each pole being simple. Since the Gamma function $\Gamma$ has simple poles at $-\N_0$, it follows that 
\[
 \Phi_m(s,\lambda,w) = \frac1{\Gamma(s)}\wt\varphi_{\lambda,w}(s)
\]
extends in $s$, for $\lambda\not=1$, to an entire function, and, for $\lambda=1$, to a meromorphic function on $\C$ with all poles contained in $\{1,\ldots, m+1\}$, all poles simple, and $m+1$ being a pole. This completes the proof.
\end{proof}

\bibliography{ap_bib}
\bibliographystyle{amsplain}

\end{document}